\documentclass[11pt,letterpaper]{article}

\usepackage{latexsym,multirow}
\usepackage{amssymb,amsmath,bm}
\usepackage{graphicx}
\usepackage[color,all,import,arrow]{xy}
\usepackage{enumerate}
\usepackage{booktabs}
\usepackage{enumitem}

\usepackage{natbib}
\usepackage{xcolor}
\usepackage[pdftex, bookmarksopen=true, bookmarksnumbered=true,
pdfstartview=FitH, breaklinks=true, urlbordercolor={0 1 0}, citebordercolor={0 0 1}, colorlinks=true, citecolor=blue!60!black]{hyperref}
\usepackage{colortbl}
\usepackage{subfigure}
\usepackage{subcaption}

\usepackage{tikz}
\usepackage{inputenc}
\usetikzlibrary{shapes,arrows,trees,fit,positioning,shapes.misc,calc}

\usepackage{dcolumn}
\newcolumntype{.}{D{.}{.}{-1}}
\newcolumntype{d}[1]{D{.}{.}{#1}}

\usepackage{theorem}
\theoremstyle{plain}
\theoremheaderfont{\scshape}
\newtheorem{assumption}{Assumption}
\newtheorem{example}{Example}
\newtheorem{definition}{Definition}
\newtheorem{corollary}{Corollary}

\newtheorem{proposition}{Proposition}
\newtheorem{theorem}{Theorem}
\newtheorem{remark}{Remark}

\newtheorem{lemma}{Lemma}
\newenvironment{proof}{\paragraph{Proof:}}{\hfill$\square$}

\newcommand{\argmin}{\operatornamewithlimits{argmin}}
\newcommand{\qed}{\hfill \ensuremath{\Box}}
\newcommand{\ind}{\mbox{$\perp\!\!\!\perp$}}

\usepackage{rotating}

\usepackage[compact]{titlesec}

\allowdisplaybreaks

\newcommand\spacingset[1]{\renewcommand{\baselinestretch}%
  {#1}\small\normalsize}

\newcommand{\blind}{0}

\newcommand{\cX}{\mathcal{X}}

\newcommand{\ba}{\bm{a}}
\newcommand{\bA}{\bm{A}}

\newcommand{\bp}{\bm{p}}
\newcommand{\bq}{\bm{q}}
\newcommand{\bv}{\bm{v}}
\newcommand{\bx}{\bm{x}}
\newcommand{\bX}{\bm{X}}

\newcommand{\cZ}{\mathcal{Z}}

\newcommand{\cD}{\mathcal{D}}

\newcommand{\cY}{\mathcal{Y}}
\newcommand{\bY}{\bm{Y}}

\newcommand{\by}{\bm{y}}

\newcommand{\add}{\textsc{Add}}

\newcommand{\std}{\textsc{Std}}

\newcommand{\cf}{\textsc{Cf}}
\newcommand{\exa}{\textsc{Exa}}

\newcommand{\E}{\mathbb{E}}
\newcommand{\R}{\mathbb{R}}

\newcommand{\bw}{\bm{w}}
\newcommand{\cP}{\mathcal{P}}
\newcommand{\cQ}{\mathcal{Q}}

\newcommand{\bC}{\bm{C}}

\usepackage{dsfont}
\usepackage{caption}
\usepackage{tikz-cd}
\usetikzlibrary{arrows.meta}

\newcommand{\lt}{\left}
\newcommand{\rt}{\right}
\newcommand{\indicator}[1]{\mathds{1}{\left\{ #1 \right\}}}

\newcommand{\Eb}[1]{\mathbb{E}{\left[ #1 \right]}}

\newcommand{\abs}[1]{\left|#1\right|}

\begin{document} 

\newcommand{\tit}{Statistical Decision Theory with Counterfactual Loss}
%
%
\spacingset{1.25}

\if0\blind

{\title{\bf\tit\thanks{We thank Iav Bojinov, Peter Buisseret, Nathan
      Cheng, Neil Christy, Johann Gaebler, Anvit Garg, D. James
      Greiner, Amanda Kowalski, Wenqi Shi, Sooahn Shin, Tomasz
      Strzalecki, Davide Viviano, Th\'{e}o Voldoire, and an anonymous
      reviewer of IQSS's rapidPeer for insightful discussions.}}

\author{Benedikt Koch\thanks{Ph.D. candidate, Department of Statistics, Harvard University. 33 Oxford Street, Cambridge MA 02138. Email: \href{mailto:benedikt_koch@g.harvard.edu}{benedikt\_koch@g.harvard.edu} URL: \href{https://benediktjkoch.github.io}{https://benediktjkoch.github.io}} \and Kosuke
  Imai\thanks{Professor, Department of Government and Department of
    Statistics, Harvard University.  1737 Cambridge Street,
    Institute for Quantitative Social Science, Cambridge MA 02138, U.S.A.
    Email: \href{mailto:imai@harvard.edu}{imai@harvard.edu} URL:
    \href{https://imai.fas.harvard.edu}{https://imai.fas.harvard.edu}}
}

\date{
\today
}

\maketitle

}\fi

\if1\blind
\title{\bf \tit}

\maketitle
\fi

\pdfbookmark[1]{Title Page}{Title Page}

\thispagestyle{empty}
\setcounter{page}{0}
         
\begin{abstract}
  Many researchers apply classical statistical decision theory to
  evaluate treatment choices and learn optimal policies. However,
  because this framework relies solely on realized outcomes under
  chosen actions and ignores counterfactuals, it cannot assess the
  quality of a decision relative to feasible alternatives at the unit
  level, which is an important requirement in some settings. For
  example, in pretrial bail decisions, a judge must balance crime
  prevention upon release against the risk of imposing unnecessary
  burdens on arrestees.  A central challenge in this framework is
  identification: since only one potential outcome is observed per
  unit, counterfactual risk is typically not identifiable. We show
  that, under strong ignorability, counterfactual risk is identifiable
  if and only if the loss is additive in the potential outcomes.  We
  further demonstrate that additive counterfactual losses can yield
  treatment recommendations that differ from those based on standard
  losses when more than two treatment options are available. We show
  that additive counterfactual losses capture not only decision
  accuracy but also decision difficulty, whereas standard losses
  reflect accuracy alone. Finally, we introduce a symbolic linear
  inverse program that determines whether a given counterfactual loss
  yields an identifiable risk, without requiring data.

\end{abstract}


\newpage
\section{Introduction}

Statistical decision theory, pioneered by \citet{wald1950statistical},
formulates decision-making as a game against nature. A decision-maker
selects a decision $D=d$ in the face of an unknown state of nature
$\theta$, and incurs a prespecified {\it loss} (or negative utility)
$\ell(d; \theta)$, which quantifies the consequence of the chosen
decision.
A better decision
corresponds to a lower loss.  Given covariates $\bX$, we can define a
decision rule $D = \pi(\bX)$ that maps observed data to decisions. The
performance of a decision rule is commonly evaluated by its {\it
  risk}, that is, the expected loss:
\begin{equation*}
R(\pi; \theta, \ell) \ := \ \E\left[\ell(\pi(\bX); \theta)\right],
\end{equation*}
where the expectation is taken over the distribution of $\bX$. Since
$\theta$ is unknown, the true risk is typically unidentifiable.  This
motivates the use of optimality criteria such as admissibility,
minimax, minimax-regret, and Bayes risk
\citep{wald1950statistical,savage1972foundations,berger1985statistical}.
Statistical decision theory also provides a foundational framework for
classical estimation theory \citep[e.g.,][]{lehmann2006theory} and
hypothesis testing \citep[e.g.,][]{lehmann1986testing}.

In a series of influential papers, \citet{manski2000identification,
  manski2004statistical, manski2011choosing} applies this framework to
the problem of treatment choice.  In its simplest form, the goal is to
choose a treatment $D$ to minimize the expected loss associated with
an outcome of interest $Y$.  Because the treatment may influence the
outcome, the loss $\ell(d; Y(d))$ depends on the potential outcome
$Y(d)$ under the chosen treatment $d \in \cD$ and possibly the
treatment itself, where $\cD$ denotes the support of decision $D$
\citep{neyman1923application,rubin1974estimating}. As before, we can
define a treatment rule $D=\pi(\bX)$ based on observed covariates, and
evaluate its performance via the following risk:
\begin{equation}
  R(\pi; \ell) \ = \ \Eb{\ell(\pi(\bX); Y(\pi(\bX)))}, \label{eq:standard_risk}
\end{equation}
where the expectation is taken over the joint distribution of $\bX$
and $Y(\pi(\bX))$. Since only one potential outcome is observed per unit, i.e.,
$Y=Y(D)$, the risk is generally unidentifiable unless the treatment
rule coincides with the one used to generate the data.  Consequently,
additional assumptions, such as the strong ignorability of treatment
assignment \citep{rosenbaum1983central}, are required to identify the
risk and learn optimal treatment rules
\citep[e.g.,][]{manski2000identification, kitagawa2018should,
  athey2021policy}.

In this paper, we generalize the standard framework of statistical
decision theory for treatment choice by introducing a broader class of
loss functions. Specifically, we consider counterfactual losses, which
depend on the full set of potential outcomes rather than the outcome
under the selected treatment alone.  That is, the loss takes the form
of $\ell(d; \{Y(d^\prime)\}_{d^\prime\in\cD})$.  We refer to this as
a {\it counterfactual loss} because it evaluates each decision in
light of the alternative outcomes that could have occurred had a
different treatment been chosen.  Given this generalized loss, we
define the {\it counterfactual risk} for measuring the performance of
treatment rule $\pi$ as:
\begin{equation}
  R(\pi; \ell) \ = \ \E[\ell(\pi(\bX); \{Y(d)\}_{d \in \cD}) ], \label{eq:counter_risk}
\end{equation}
where the expectation is taken over the distribution of covariates
$\bX$ and all potential outcomes $\{Y(d)\}_{d\in \cD}$. 

Counterfactual losses raise two foundational questions: coherence and
identification. Coherence asks whether decision making based on
counterfactual losses satisfies basic rationality requirements such as
transitivity. In a companion paper
\citep{koch2026axiomaticfoundationdecisionscounterfactual}, we
establish an axiomatic foundation for counterfactual loss by showing
that the preferences it induces are coherent and satisfy a
generalization of the von Neumann-–Morgenstern axioms
\citep{NeumannMorgenstern1944}.  We briefly summarize this result in
Section~\ref{sec:preferences_of_counterfactual_loss}.

In this paper, we take the coherence of counterfactual loss as given
and study its identification. In contrast to the standard risk in
Equation~\eqref{eq:standard_risk}, even under strong ignorability, the
counterfactual risk in Equation~\eqref{eq:counter_risk} is generally
not identifiable, due to its dependence on the joint distribution of
potential outcomes.  We derive a structural condition on the
counterfactual loss that enables the identification of counterfactual
risk (Section~\ref{sec:identification}). Specifically, we consider a
class of counterfactual losses that are {\it additive} in the
potential outcomes. We show that under strong ignorability of
treatment assignment, additivity is both necessary and sufficient for
identifying counterfactual risk {\it differences}, and hence for
identifying an optimal policy within a given policy class.
Importantly, our results extend beyond the binary treatment and
outcome cases that have been a focus of the existing literature.

We next examine how restrictive the additivity condition is
(Section~\ref{sec:expressiveness_add}). When decisions are binary,
i.e., $|\cD| = 2$, we show that every additive counterfactual loss
admits an equivalent standard loss that induces the same ranking over
policies and hence the same optimal decision rule. By contrast, when
the decision problem involves more than two treatment options
($|\cD| > 2$), additive counterfactual losses can induce rankings and
optimal policies that cannot be replicated by any standard loss. This
divergence arises because additive counterfactual losses capture both
the \textit{accuracy} and the \textit{difficulty} of a decision,
whereas standard losses reflect accuracy alone. Furthermore, we show
that, under binary outcomes and a monotonicity restriction, every
counterfactual loss is additive.

Finally, we formulate a symbolic linear inverse problem that offers a
practical test for whether a given loss is additive and hence its risk
identifiable (Section~\ref{sec:symbolic}). This test does not require
any data, and the solution to this linear inverse problem provides a
closed-form decomposition of counterfactual risk into identifiable
marginal distributions of potential outcomes.

\paragraph{Related Literature.}

The application of counterfactual logic to decision-theoretic problems
has appeared in several strands of the literature. One of the most
prominent examples of counterfactual loss is the notion of regret,
$\ell(d; \{Y(d^\prime)\}_{d^\prime \in \cD}) = \max_{d^\prime \in \cD}
Y(d^\prime) - Y(d)$ \citep{savage1972foundations}. Yet, because of the
aforementioned identification challenge, prior work has mostly focused
on the settings with binary treatment and binary outcome. For example,
\citet{coston2020counterfactual} and \citet{ben-michaelDoesAIHelp2024}
consider the counterfactual loss that is a function of the baseline
potential outcome $Y(0)$ alone. A more general binary case has been
studied by \citet{li2019unit,li2022unit, NEURIPS2022_ebcf1bff,
  kallus2022s,mueller2023personalized,ben-michaelPolicyLearningAsymmetric},
who address non-identifiability through partial-identification bounds
and by leveraging observational data in some cases. In addition,
\citet{christy2024startingsmallprioritizingsafety,
  christy2026countingdefiersdesignbasedmodel} circumvent
non-identifiability by evaluating decisions with a finite-sample
maximum-likelihood rule.

The most relevant work in the binary case is \citet{li2019unit} which
provides a sufficient condition for identification.  We show that this
condition follows directly from our general identification result and
it is indeed necessary for the identification of counterfactual risk.
However, our result implies that their condition can be relaxed if the
goal is to identify a counterfactual risk difference and can be
expressed as a mathematically equivalent standard loss.

Extensions to non-binary treatments have been studied by
\citet{li2024probabilities, li2024unit}.  Specifically,
\citet{li2024probabilities} derive partial-identification bounds,
while \citet{li2024unit} propose a non-constructive algorithm that
numerically tests identifiability using experimental data.  In
contrast, we provide a closed-form, analytical characterization of
counterfactual losses that is both necessary and sufficient for the
identification of counterfactual risk.  We also extend these results
to the identification of counterfactual risk difference, which
suffices for learning optimal treatment rules. Our identification
results lead to a symbolic linear inverse problem formulation that can
be used, without data, to verify whether a counterfactual risk (or the
difference between any pair of counterfactual risks) is
identifiable. We further show that in the non-binary treatment case,
an identifiable counterfactual loss need not reduce to a standard
loss.



\section{Illustrative Examples}
\label{sec:examples}

Before presenting our general identification analysis in
Section~\ref{sec:identification}, we use illustrative examples to
highlight the differences between standard and counterfactual losses. We begin with a binary decision example and then consider a trichotomous decision example. Numerical illustrations are provided in Appendix~\ref{app:numerical_illustrations}.

\subsection{Binary Decisions}\label{subsec:hippocratic_oath}

Consider a stylized setting, in which a physician must decide whether to
approve ($D = 1$) a new medical treatment for a population of at-risk
individuals, or not ($D = 0$). The outcome of interest is survival,
denoted by $Y$, where $Y = 1$ indicates survival and $Y = 0$
otherwise. For simplicity, we assume no covariates are observed.  We introduce standard, counterfactual, and additive counterfactual
losses in the context of this example.

\subsubsection{Standard Loss}

Suppose that we evaluate the physician's decision using the standard
loss $\ell^\std(D; Y(D))$, which assigns a value to each realized
decision--outcome pair $(D, Y(D))$. Let $c_1 > 0$ denote the cost of
providing treatment relative to the cost of no intervention $c_0 = 0$.
We use $\ell_y$ to denote the loss associated with the realized
outcome $Y(D)=y$. For example, if the treatment is administered and
the patient dies, the loss is given by
$\ell^\std(1;0) = \ell_0 + c_1$. If no treatment is given and the
patient survives, the loss is $\ell^\std(0;1) = \ell_1$. Since
survival is preferred, we assume $\ell_1 < \ell_0$.

This yields the standard loss function,
\begin{align}\label{eq:ex_standard_loss}
    \ell^\std(D;Y(D)) = \ell_{Y(D)} + c_D.
\end{align}
An optimal decision $d^\ast$ minimizes the standard risk,
\begin{align*}
    d^\ast = \argmin_{d \in \{0, 1\}} R(d; \ell^\std),
\end{align*}
where $R(d; \ell^\std) = \Eb{\ell^\std(d;Y(d))}$.  The classical A/B
test is a special case of this standard loss where
$\ell^\std(D; Y(D)) = -Y(D)$, reducing risk minimization to the
comparison of average survival rates.

\subsubsection{Counterfactual Loss}

The standard loss $\ell^\std$ only considers the realized outcome
under the chosen decision. In contrast, the physician may decide to use
a counterfactual loss that depends on both potential outcomes
simultaneously. This naturally leads to the concept of \emph{principal
  strata}, which classifies individuals according to their joint
potential outcomes \citep{frangakis2002principal}:
\begin{itemize}
  \item \textbf{Never survivors} $(Y(0), Y(1)) = (0, 0)$: patients who
    would not survive regardless of treatment 
  \item \textbf{Responders} $(Y(0), Y(1)) = (0, 1)$: patients who
    would survive only if given the treatment 
  \item \textbf{Harmed} $(Y(0), Y(1)) = (1, 0)$: patients who would
    survive only if not treated 
  \item \textbf{Always survivors} $(Y(0), Y(1)) = (1, 1)$: patients
    who would survive regardless of treatment 
\end{itemize}
Building on this idea, several authors have considered the following
counterfactual loss \citep{li2019unit,
  ben-michaelPolicyLearningAsymmetric,
  christy2024startingsmallprioritizingsafety}:
\begin{equation}\label{eq:asymmetric_loss}
\begin{aligned}
\ell^\cf(D; Y(0), Y(1)) & = (1-Y(0))(1-Y(1))\ell_0 + (1-Y(0))Y(1) \ell_D^\text{R}  \\
        & \qquad + Y(0)(1-Y(1))\ell_{1-D}^\text{H} + Y(0)Y(1)\ell_1 +  c_D.
\end{aligned} 
\end{equation}
This generalizes the standard loss given in
Equation~\eqref{eq:ex_standard_loss}, which is obtained as a special
case by setting $\ell_{D}^{\text{R}}=\ell_D$ and $\ell_{1-D}^{\text{H}}=\ell_{1-D}$.

Unlike the standard loss, therefore, the counterfactual loss
incorporates the information about whether the treatment would alter
the outcome. For example, responders and always survivors both incur
the same standard loss $\ell^\std(1;1) = \ell_1 + c_1$ under the
treatment, but their counterfactual losses may differ (e.g.,
$\ell_1^\text{R} + c_1 < \ell_1 + c_1$ if we reward the correct
decision).  \citet{ben-michaelPolicyLearningAsymmetric} design a
counterfactual loss that reflects a version of the Hippocratic Oath
principle of ``Do no harm'' \citep{wiens2019no}, where killing a
patient is worse than failing to provide life-saving treatment, i.e.,
$\ell_0^{\text{R}} - \ell_1^{\text{R}} < \ell_0^{\text{H}} -
\ell_1^{\text{H}}$.

Counterfactual loss can also reveal heterogeneity that standard loss
may obscure. For example, \citet{mueller2023personalized} exhibit a
setting in which men and women have the same standard risk difference,
$R(1; \ell^\std) - R(0; \ell^\std)$, yet the treatment is safe for
women but harmful for men \citep[see also][]{kallus2022s}. Deploying
such a policy could therefore be ethically problematic, yet this issue
would not be detected by standard loss (and hence by a classical A/B
test).

Despite its ability to express rich preferences, the main limitation
of counterfactual loss is identifiability since principal strata are
never observed \citep{sarvet2023perspectives,dawid2000causal}.  As a
result, the literature primarily focused on partial identification
\citep[e.g.,][]{li2019unit, li2024probabilities,
  ben-michaelPolicyLearningAsymmetric}.

\subsubsection{Additive Counterfactual Loss}\label{sec:add_counterfactual_loss}

\begin{table}[t]
\centering 
\begin{tabular}{l|c|c|c|} 
  \multicolumn{2}{c}{}  & \multicolumn{2}{c}{\textbf{Decision}} \\\cline{3-4}
      \multicolumn{2}{c|}{}  & \multirow{2}{*}{Negative $(D = 0$)}  & \multirow{2}{*}{Positive $(D = 1$)}\\
      \multicolumn{2}{c|}{} & \multicolumn{1}{c|}{} & \multicolumn{1}{c|}{} \\ \cline{2-4}
  & \multirow{2}{*}{Negative ($Y(0) = 0$)} & \cellcolor[HTML]{F25454}{\color[HTML]{333333} False Negative (FN)} & \cellcolor[HTML]{B9EBB7}True Positive \\
 &   & $\ell_0$ &  $\tilde\ell_0 + c_1$ \\ \cline{2-4}
  &  \multirow{2}{*}{Positive ($Y(0) = 1$)} & \cellcolor[HTML]{5BCB5C}True Negative  & \cellcolor[HTML]{FFCCC9}{\color[HTML]{333333} False Positive (FP)}\\
 &   & $\ell_1$ &  $\tilde\ell_{1}+c_1$ \\ \cline{2-4} 
 \multirow{-2}{*}{\textbf{Potential Outcomes}} & \multirow{2}{*}{Negative ($Y(1) = 0$)} & \cellcolor[HTML]{B9EBB7}True Negative & \cellcolor[HTML]{F25454}{\color[HTML]{333333} False Positive (FP)} \\
 &   & $\tilde\ell_{0}$ &  $\ell_{0}+c_1$ \\ \cline{2-4}
 &  \multirow{2}{*}{Positive ($Y(1) = 1$)} & \cellcolor[HTML]{FFCCC9}{\color[HTML]{333333} False Negative (FN)} & \cellcolor[HTML]{5BCB5C}True Positive\\
 &   & $\tilde\ell_{1}$ &  $\ell_{1}+c_1$ \\ \cline{2-4}
 \cline{2-4} 
\end{tabular}
\caption{Confusion matrix indexed by $Y(d)=y_d$ and decision $D=d$. Each cell is assigned the corresponding component of the additive loss. Darker colors indicate realized outcomes, while lighter colors indicate counterfactual outcomes.}
\label{tbl:confusion}
\end{table}

We propose an alternative strategy to non-identifiability by
restricting the class of counterfactual losses to those that are
additive in the potential outcomes.  We now provide an example of
additive counterfactual loss by extending a counterfactual loss
considered in the literature \citep{coston2020counterfactual,
  ben-michaelDoesAIHelp2024} that depends on the baseline potential
outcome $Y(0)$ alone. Under this formulation, the decision is a false
negative if the physician does not provide the treatment and the
patient dies, $(D, Y(0)) = (0, 0)$, in which case the loss is $\ell_0$
since $Y(0)$ is the realized outcome in this case. Conversely, a false
positive occurs when the physician provides the treatment even though
the patient would have survived without it, $(D, Y(0)) = (1, 1)$. In
this case, $Y(0)$ is counterfactual, and we denote the associated loss
as $\tilde\ell_1$.  Under this setup,
\citet{ben-michaelDoesAIHelp2024} propose the following counterfactual
classification loss:
\begin{equation}
\ell(D; Y(0)) = (1 - D)(1 - Y(0))\ell_0 + D \cdot Y(0) \tilde\ell_1 + c_D. \label{eq:classification_loss}
\end{equation}

As illustrated in Table~\ref{tbl:confusion}, this counterfactual loss
can be further generalized. First, we may also incorporate losses
based on the potential outcome under treatment, covering cases where
$(D, Y(1)) = (0,1)$ (false negative) or $(D, Y(1)) = (1,0)$ (false
positive). Second, we may consider rewards (negative losses) for
correct decisions: $(D, Y(0)) = (1,0)$ and $(D, Y(1)) = (1,1)$ are
true positives, while $(D, Y(0)) = (0,1)$ and $(D, Y(1)) = (0,0)$ are
true negatives. Letting $\tilde\ell_0$ denote the loss incurred when a
patient would have died under the alternative (unchosen) decision, a
more general counterfactual loss can be written as:
\begin{equation}
\ell^\add(D; Y(0), Y(1)) = \ell_{Y(D)} + \tilde\ell_{Y(1-D)} + c_D. \label{eq:classification_loss_general}
\end{equation}
This formulation is additive in the potential outcomes.  While it does
not fully capture stratum-specific interactions, additive
counterfactual loss does account for the notion of counterfactual
regret. For example, if $\tilde\ell_0 < \tilde\ell_1$, this reflects a
regret for not choosing the decision that would save the patient.

Despite these differences, as shown later, in the binary decision
case, every additive counterfactual loss has an equivalent standard
loss.  For example, the additive counterfactual loss in
Equation~\eqref{eq:classification_loss_general} is equivalent to
$\ell^{\textsc{Equ}}(D; Y(D)) = \ell_{Y(D)} - \tilde\ell_{Y(D)} +
c_D$, because we have
\[
    \ell^\add(D; Y(0), Y(1)) - \ell^{\textsc{Equ}}(D; Y(D)) = \tilde\ell_{Y(1-D)} + \tilde\ell_{Y(D)} = \tilde\ell_{Y(0)} + \tilde\ell_{Y(1)},
\]
which does not depend on the decision.  Hence,
$R(d; \ell^\add)$ and $R(d; \ell^{\textsc{Equ}})$ induce the same
ordering of treatment rules.

\subsection{Trichotomous Decision}\label{subsec:trichotomous_decision}

We next move beyond binary decisions and illustrate counterfactual additive loss in the case of trichotomous
decisions. We introduce a different stylized example taken from
criminal justice \citep[see e.g.,][]{imai2023experimental}.  Consider a first appearance hearing, where a judge
must decide whether to release an arrestee on their own recognizance
or ROR ($D=0$), impose supervision (e.g., monitoring via an electronic
device) as a condition of release ($D=1$), or detain the arrestee
until trial ($D=2$).  As before, let $Y$ be the binary outcome of
interest, where $Y=0$ indicates a negative outcome (e.g., the arrestee
commits a crime before trial) and $Y=1$ represents a positive outcome
(e.g., no new crime). For simplicity, we assume that detention
prevents crime, i.e., $Y(2) = 1$, though this assumption can be
relaxed. We also assume no covariates are observed.

\subsubsection{Standard Loss}

Suppose a judge seeks to balance the social cost of crime against the costs of supervision and incarceration, which are borne at least in part by arrestees. 
Let $\ell_y$ denote the
loss associated with outcome $Y=y$. Since we would like to avoid
crime, we assume $\ell_1 < \ell_0$.  Let $c_d$ be the cost of decision
$d$. The cost of detention is assumed to be greater than that of
supervision, which in turn is larger than the cost of ROR, i.e.,
$c_0 < c_1 < c_2$.  The corresponding standard loss is given by,
\[
\ell^\std(D; Y(D)) = \ell_{Y(D)} + c_D.
\]

\subsubsection{Counterfactual Loss}

This standard loss does not account for two important considerations a
judge must weigh: \emph{harm to the public} and \emph{harm to the
  arrestee}. Harm to the public arises when a judge's decision leads
to a crime, whereas harm to the arrestee occurs when the decision
imposes unnecessary costs through supervision or detention. It is
natural to assume that a judge seeks to minimize both types of
harm. The following counterfactual loss incorporates these two notions of harm, 
\begin{equation}\label{eq:tri_counterfactual_loss}
\begin{aligned}
  & \ell^\cf(D; Y(0), Y(1), Y(2)) \\
  = \ & \ell_{Y(D)} + c_D + \indicator{D = 0}\, (1 - Y(0)) Y(1) p_0 +
  \indicator{D = 1}\, Y(0)(1 - Y(1)) p_1 + \sum_{k < D} r_{D,k} Y(k),
\end{aligned}
\end{equation}
where $p_d$ is the additional penalty for causing harm to the public
with decision $D=d$ for $d=0,1$ (recall that $D=2$ is assumed to lead
to no crime), and $r_{d,k}$ represents the regret for imposing an
unnecessarily harsh decision $d$ when a more lenient decision $k < d$
would lead to no crime.

\subsubsection{Additive Counterfactual Loss}

Since the additional penalty term for the harm to the public depends
on joint potential outcomes, the counterfactual risk based on
Equation~\eqref{eq:tri_counterfactual_loss} is not identifiable.
However, the regret term for the harm to the arrestee does not involve
joint potential outcomes and hence can be incorporated into an
additive counterfactual loss in the following manner,
\begin{align}\label{eq:tri_additive_loss}
    \ell^\add(D; Y(0), Y(1), Y(2)) &= \ell_{Y(D)} + c_D + \sum_{k < D} r_{D,k} Y(k).
\end{align}

\section{Identification under Strong Ignorability}
\label{sec:identification}

We now formally define counterfactual loss and risk, and introduce the
class of additive counterfactual losses, which depend additively on
the potential outcomes. We show that additivity is both necessary and
sufficient for identification.

\subsection{Setup}\label{sec:setup}

Consider a setting in which, for a unit with pre-treatment covariates
\(\bX\in\cX\), a decision \(D\in\cD:=\{0,1,\ldots,K-1\}\) is chosen
and an outcome \(Y\in\cY:=\{0,1,\ldots,M-1\}\) is observed. Here,
\(\cX\) is an arbitrary measurable space, and we assume $K, M \geq
2$. Let \(Y(d)\in\cY\) denote the potential outcome under decision
\(d\in\cD\). The vector of potential outcomes and covariates has an unknown law,
\[
    (Y(0),\ldots,Y(K-1),\bX)
    \sim P_{Y(\cD),\bX} \in \Delta(\cY^\cD\times\cX),
\]
where \(\Delta(\cY^\cD\times\cX)\) denotes the set of
all probability measures on \(\cY^\cD\times\cX\). We view this law as
describing an unknown superpopulation, while treating potential
outcomes as fixed for a given unit
\citep{neyman1923application,rubin1974estimating}.

The observed data are generated under the realized decision $D$ via a
\emph{treatment assignment mechanism}
$e:\cY^{\cD}\times\cX\to\Delta(\cD)$, where $\Delta(\cD)$ denotes the
set of probability distributions on $\cD$. Given
$P_{Y(\cD),\bX} \in \Delta(\cY^\cD\times\cX)$, define the induced
\emph{joint} law $P$ of $(D, Y(0), \ldots, Y(K-1), \bX)$ as
\begin{align}\label{eq:joint_law}
    P(D = k, Y(\cD) = \by, \bX \in A) = \int_{A} e(k; \by, \bx) \cdot P_{Y(\cD) \mid \bX}(Y(\cD) = \by \mid \bX = \bx) dP_{\bX}(\bX=\bx),
\end{align}
for all $k \in \cD, \by = (y_0, \ldots, y_{K-1}) \in \cY^\cD$ and measurable $A \subseteq \cX$, where $Y(\cD) = (Y(0), \ldots, Y(K-1))$.  

We study the identifiability of counterfactual risk under the standard consistency and strong ignorability assumptions \citep{rosenbaum1983central} on the joint law $P$.
\begin{assumption}[Consistency]\label{ass:consistency}
The observed outcome satisfies: 
$Y = Y(D)$.
\end{assumption}

\begin{assumption}[Strong Ignorability]\label{ass:ignorability} The observed decision satisfies:
\begin{enumerate}[label=(\alph*)]
\item \label{ass:unconfoundedness} Unconfoundedness: $D \ind \{Y(d)\}_{d \in \cD} \mid \bX$; equivalently,
  $e(d;\by,\bx)=e(d;\bx)$ for every $d\in\cD$,
  $P_{Y(\cD),\bX}$-almost surely.
\item \label{ass:overlap} Overlap: $e(d;\bx)>0$ for every $d\in\cD$, $P_{\bX}$-almost surely.
\end{enumerate}
\end{assumption}
Assumption~\ref{ass:consistency} links the observed outcome to the
potential outcomes through the realized decision.
Assumption~\ref{ass:ignorability}\ref{ass:unconfoundedness} implies
that, conditional on covariates, the observed decision $D$ is
independent of the potential outcomes while
Assumption~\ref{ass:ignorability}\ref{ass:overlap} requires that each
decision occurs with positive probability conditional on covariates.

Given $P_{Y(\cD),\bX} \in \Delta(\cY^\cD \times \cX)$, define
$\mathcal{E}(P_{Y(\cD),\bX})$ as the set of assignment mechanisms that
satisfy Assumption~\ref{ass:ignorability} with respect to
$P_{Y(\cD),\bX}$.
Define the following set of all joint laws compatible with our assumptions, 
\begin{align*}
    \cP := \{P: P = e \cdot P_{Y(\cD),\bX}, P_{Y(\cD),\bX} \in
  \Delta(\cY^\cD \times \cX), e \in \mathcal{E}(P_{Y(\cD), \bX})\},
\end{align*}
where $P=e\cdot P_{Y(\cD),\bX}$ denotes the joint law induced by $(e, P_{Y(\cD),\bX})$ through Equation~\eqref{eq:joint_law}.
Note that these assumptions restrict the assignment mechanism but place no restrictions on the law of the potential outcomes and covariates.

Although the observed data are generated under the realized decision
\(D\), our aim is to evaluate a given, possibly randomized
\emph{policy} \(\pi:\cX\to\Delta(\cD)\), which may differ from the
observed assignment mechanism $e$. The policy \(\pi\) induces a
decision \(D^\pi\) according to
\[
    D^\pi \mid \bX=\bx \sim \pi(\cdot;\bx).
\]
Similar to the assignment mechanism, we assume that any policy
randomization is independent of the unit's potential outcomes
conditional on the covariates.  That is, if $D^\pi = g(\bX, U)$, where
\(U\sim \mathrm{Unif}(0,1)\), then
$U \ind (Y(0),\ldots, Y(K-1)) \mid \bX$.

\subsection{Counterfactual Loss and Risk}\label{subsec:counterfactual_loss_risk}

We define counterfactual losses, which depend on the full vector of potential outcomes, and contrast them with standard losses, which depend only on the realized outcome under the chosen decision. A counterfactual loss induces a counterfactual risk that depends on the unknown joint distribution of potential outcomes, whereas a standard loss induces a standard risk that depends only on the observed marginal distribution of the realized outcome.

We define a loss (or negative utility) associated with a decision as a
function of the covariates and all potential outcomes, where a lower
loss value indicates a better decision. We refer to this as the 
\textit{counterfactual loss}, as it depends on the full vector of
potential outcomes. 
\begin{definition}[Counterfactual Loss]\label{def:loss}
  A counterfactual loss is a measurable function
  \( \ell : \cD \times \cY^\cD \times \cX \to \R \), where
  \( \ell(d; \by, \bx) \) denotes the loss of
  choosing decision \( D = d \) for a unit with potential
  outcomes \( Y(\cD) = \by \)
  and covariates \( \bX = \bx \).
\end{definition}
A subclass of counterfactual losses are \emph{standard losses}, which depend only on the realized outcome under the chosen decision, i.e., $\ell^{\std}(d;\by,\bx)=\ell^{\std}(d;y_d,\bx)$.

Since our goal is to study identification uniformly over all $P \in \mathcal{P}$, we impose a boundedness condition on the loss to accommodate the fact that $\cX$ need not be finite.
\begin{assumption}[Bounded Counterfactual Loss]\label{ass:bounded_loss}
    The counterfactual loss $\ell$ satisfies
    \[
        \|\ell\|_{\infty} = \sup_{d\in \cD, \by \in \cY^\cD, \bx \in \cX}|\ell(d; \by, \bx)| < \infty.
    \]
\end{assumption}

Next, we define the \textit{counterfactual risk}, or the expected
counterfactual loss, which serves as the primary estimand of
interest. We also define the \textit{conditional counterfactual risk},
which conditions on the observed covariates.

\begin{definition}[Counterfactual Risk and Conditional Counterfactual Risk] \label{def:risk}
    Given a joint law $P \in \mathcal{P}$ and a counterfactual loss $\ell$, the counterfactual risk $R_P(\pi; \ell)$ of policy $\pi: \cX \to \Delta(\cD)$ is given by,
    \begin{align*}
         R_P(\pi; \ell) &:= \E_{P^\pi}[\ell(D^\pi; Y(0), \ldots, Y(K-1), \bX)] = \E_{P_{\bX}}[R_{P_{Y(\cD) \mid \bX}}(\pi; \ell)],
    \end{align*}
    where,
    \begin{align*}
    R_{P_{Y(\cD) \mid \bX = \bx}}(\pi; \ell) := \sum_{d \in \cD} \sum_{\by \in \cY^\cD} \ell(d; \by, \bx) \cdot \pi(d; \bx) \cdot P(Y(0) = y_0, \ldots, Y(K-1) = y_{K-1} \mid \bX = \bx),
\end{align*}
is the conditional counterfactual risk defined almost surely with
respect to $P_{\bX}$.
\end{definition}
Definition~\ref{def:risk} evaluates a fixed policy by its expected counterfactual loss. The expectation is taken with respect to the joint law $(D^\pi, Y(0),\ldots,Y(K-1),\bX) \sim P^\pi$ where 
\[
    P^\pi(D^\pi=d,Y(\cD)=\by,\bX\in A)
    =
    \int_A \pi(d;\bx) \cdot
    P_{Y(\cD)\mid \bX}(Y(\cD)=\by\mid \bX=\bx)
    \,dP_{\bX}(\bX = \bx).
\]
The policy $\pi$ determines the conditional distribution of the
decision given the covariates, while $P_{Y(\cD),\bX}$ specifies the
(unknown) joint distribution of potential outcomes and covariates. We
therefore view $P_{Y(\cD),\bX}$ as the unknown \emph{state of nature}
from statistical decision theory \citep[e.g.,][]{wald1950statistical,
  berger1985statistical, manski2004statistical, STOYE200970}. Note that $P$ and
$P^\pi$ have the same marginal law of potential outcomes and
covariates, i.e., $P_{Y(\cD),\bX} = P^\pi_{Y(\cD),\bX}$.  Similarly,
we refer to the risk induced by a standard loss as a \emph{standard
  risk}.

Following the literature \citep[e.g.,][]{li2019unit,coston2020counterfactual,
mueller2023personalized,ben-michaelPolicyLearningAsymmetric}, we treat $\pi$ as fixed rather than learned from observational data. If instead the policy was estimated from a sample \(\mathcal S_n=\{(D_i, Y_i, \bX_i)\}_{i=1}^n\sim Q_P^n\), where
\(Q_P\) denotes the observable law induced by $P \in \cP$, we would additionally average over the realization of the sample
\begin{align}\label{eq:sampling_risk}
    \mathcal R_{P}(\hat \pi;\ell) =  \E_{Q_P^n}[R_P(\hat \pi(\mathcal S_n); \ell)],
\end{align}
where $\hat \pi(\mathcal{S}_n): \cX \to \Delta(\cD)$ is a
\emph{statistical decision rule} \citep[Section
1.3]{berger1985statistical} that maps a sample to a policy.  We leave
the study of this finite sample counterfactual risk to future work
\citep[see e.g.,][for this finite sample standard
risk]{berger1985statistical,manski2004statistical,
  stoye2011statistical,kitagawa2018should}.

\subsection{Additivity}\label{subsec:additivity}

The main obstacle to identifying counterfactual risk is that under a counterfactual loss the risk may depend on the joint distribution of all potential outcomes, even though only a single realized outcome is observed. To overcome this identification problem we consider a class of loss functions that depends additively on the potential outcomes.

\begin{definition}[Additive Counterfactual Loss] \label{def:additive}
  A counterfactual loss is \emph{additive} if it admits the
  following decomposition,
\begin{align*}
        \ell^\add(d;\by, \bx) := \varpi(\by, \bx) + \sum_{k \in \cD} \omega_k(d, y_k,\bx),
    \end{align*}
    where $\omega_k: \cD \times \cY \times \cX \rightarrow \R$ and $\varpi: \cY^\cD \times \cX \rightarrow \R$ are bounded, measurable functions. We call $\omega_k$ the weight function and $\varpi$ the
    intercept function.
\end{definition}

The additive counterfactual loss equals the sum of two terms; the
intercept function does not depend on the decision, while
the weight function does.  Moreover, each weight function $\omega_k$
depends only on the corresponding potential outcome given the
covariate, whereas the intercept function may include all possible
interactions of potential outcomes. Although we cannot identify the
intercept function, we can ignore it when comparing alternative
decision rules because it does not depend on the decision. We
generally refer to the counterfactual risk induced by an additive
counterfactual loss as \emph{additive counterfactual risk}.

We now illustrate Definition~\ref{def:additive} by revisiting the examples shown in Section~\ref{sec:examples}.
\begin{example}[Counterfactual Classification
  Loss] \label{ex:counterfactual_classification_loss} The
  counterfactual loss function given in
  Equation~\eqref{eq:classification_loss} is additive in the potential
  outcomes with $\omega_0(0, 0) = \ell_0$,
  $\omega_0(1, 1) = \tilde\ell_1+c_1$, $\omega_1(d,y_1)=0$ for
  $d,y_1\in\{0,1\}$, and $\varpi(\by)=0$. An extension of this
  counterfactual loss given in
  Equation~\eqref{eq:classification_loss_general} is also additive
  with $\omega_d(d, y_d) = \ell_{y_d} + c_d$,
  $\omega_{1-d}(d, y_{1-d}) = \tilde\ell_{y_{1-d}} $ for
  $d,y \in\{0,1\}$, and $\varpi(\by)=0$.
\end{example}

\begin{example}[Asymmetric Counterfactual
  Loss] \label{ex:hippocratic_oath} The asymmetric counterfactual loss
  given in Equation~\eqref{eq:asymmetric_loss} is an example of
  non-additive loss whenever $\ell_0^\text{R} - \ell_1^\text{R} \neq \ell_0^\text{H} - \ell_1^\text{H}$
  (see Equation~\eqref{eq:constant} in Section~\ref{subsec:principal_strata} for details). When
  $\ell_0^\text{R} - \ell_1^\text{R} = \ell_0^\text{H} - \ell_1^\text{H} = \kappa$, the loss is
  additive with $\omega_d(d, y_d) = \ell_{y_d} + c_d$,
  $\omega_{1-d}(d, y_{1-d}) = \{\kappa - (\ell_0 - \ell_1)\} y_{1-d}$, and
  $\varpi(\by) = (\ell_1^{\text{H}} - \ell_1) y_0 (1-y_1) +
  (\ell_1^{\text{R}} - \ell_1) (1-y_0)y_1 - \{\kappa - (\ell_0 -
  \ell_1)\} y_0 y_1$. Note that $\kappa - (\ell_0 - \ell_1)$ is the
  difference in loss between units for whom the treatment matters and
  those for whom it does not.
\end{example}

\begin{example}[Trichotomous decision]
  The counterfactual loss for individual harm, given in
  Equation~\eqref{eq:tri_additive_loss}, is additive in the potential
  outcomes with $\omega_d(d, y_d) = \ell_{y_d} + c_d$,
  $\omega_k(d, y_k) = r_{d,k} y_k \indicator{k < d}$ for
  $d, k \in\{0,1, 2\}, y_d, y_k \in \{0, 1\}$, and $\varpi(\by)=0$. By
  contrast, the extended counterfactual loss given in
  Equation~\eqref{eq:tri_counterfactual_loss} that also incorporates
  harm to the public is additive only if $p_0 = p_1 = 0$, in which case it reduces to
  Equation~\eqref{eq:tri_additive_loss}.
\end{example}

\subsection{Nonparametric Identification}\label{sec:nonparametric_identification}

We now study when counterfactual risk is identifiable under strong
ignorability. We show that additive losses are sufficient to identify
counterfactual risk differences and are also necessary for uniform
identification for every $P\in \cP$.

The fundamental problem of causal inference is that only one potential
outcome is observed for each unit \citep{holland1986}. Although the
full data are governed by $(D,Y(0),\ldots,Y(K-1),\bX) \sim P$ for some $P\in\cP$,
we observe only $(D,Y,\bX)$, where $Y=Y(D)$ by
Assumption~\ref{ass:consistency}. For $P\in\cP$, let $Q_P$ denote the
induced observed law
\[
    Q_P(D=k,Y=y,\bX\in A)
    :=
    P(D=k,Y(k)=y,\bX\in A).
\]

We adapt the following definition of identification from
\citet{basse2020generaltheoryidentification} to our setting. A more
general definition is given in
Lemma~\ref{lem:equivalence_of_identifiability}
(Appendix~\ref{app:equivalence_of_identifiability}).

\begin{definition}[Identifiability]\label{def:identifiability}
    Let $\theta: \cP \to \R$ be an estimand. We say that $\theta$ is \emph{identifiable} if every pair $P_1, P_2 \in \cP$ that induce the same observed law, $Q_{P_1} = Q_{P_2}$,  also have the same value, $\theta(P_1) = \theta(P_2)$.
\end{definition}

According to Definition~\ref{def:identifiability}, an estimand is identifiable if its value is invariant across all admissible joint distributions that induce the same observable law. In our setting, the counterfactual risk $R_P(\pi; \ell)$ is identifiable if and only if any two admissible joint laws $P_1, P_2 \in \cP$, with the same induced observed law $Q_{P_1} = Q_{P_2}$, have the same counterfactual risk $R_{P_1}(\pi;\ell)=R_{P_2}(\pi;\ell)$.


Equivalently, by Lemma~\ref{lem:equivalence_of_identifiability} (Appendix~\ref{app:equivalence_of_identifiability}), $R_P(\pi;\ell)$ is identifiable if and only if it can be expressed as a function of the observable distribution, i.e., $R_P(\pi; \ell) = h_{\pi, \ell}(Q_P)$. 
Furthermore, as formalized and shown in Lemma~\ref{lem:risk_identifiable_iff_conditional_risk_identifiable} (Appendix~\ref{app:risk_identifiable_iff_conditional_risk_identifiable}), the counterfactual risk $R_P(\pi;\ell)$ is identifiable with respect to $Q_P$ if and only if the conditional counterfactual risk $R_{P_{Y(\cD) \mid \bX = \bx}}(\pi;\ell)$ is identifiable with respect to $Q_P(D, Y \mid \bX)$.

While estimands that depend on the joint distribution of multiple
potential outcomes are generally not identifiable,
Assumption~\ref{ass:consistency}~and~\ref{ass:ignorability} imply that
the following conditional distribution is identified,
\[
    P(Y(k)=y\mid\bX)
    = P(Y = y \mid D = k, \bX) = 
    Q_P(Y=y\mid D=k,\bX),
\]
for every \(k\in\cD\) and \(y\in\cY\). This observation suggests that the additivity assumption given in Definition~\ref{def:additive} can facilitate identification.

We first show that for an additive loss, we can identify the {\it
  difference} in conditional counterfactual risk between any two
decision rules even though each conditional counterfactual risk is
only identifiable up to an unknown additive constant.  In other words,
the conditional counterfactual risk of a policy $\pi$ can be
decomposed into an identifiable part and an unidentifiable part, and
the latter does not depend on $\pi$.

\begin{theorem}[The Conditional Counterfactual Risk under the Additive Loss] \label{thm:additivity_marginal_risk}
    Fix a joint law $P \in \cP$ and a policy $\pi: \cX \to \Delta(\cD)$. Suppose that a counterfactual loss function $\ell^\add$ takes the additive form given in Definition~\ref{def:additive}.  Then, under Assumptions~\ref{ass:consistency}--\ref{ass:bounded_loss}, the conditional counterfactual risk can be decomposed as,
    \begin{align*}
        &R_{P_{Y(\cD) \mid \bX = \bx}}(\pi; \ell^\add) = \sum_{d\in \cD} \sum_{k\in \cD} \sum_{y\in \cY} \omega_k(d, y, \bx) \cdot \pi(d; \bx) \cdot P(Y = y\mid D = k, \bX = \bx) + C_{P_{Y(\cD) \mid \bX}}(\bx).
    \end{align*}
    where 
    $$C_{P_{Y(\cD) \mid \bX}}(\bx) =  \sum_{\by \in \cY^\cD} \varpi(\by, \bx) \cdot P(Y(\cD) = \by\mid \bX = \bx),$$ 
    Thus, the difference of conditional counterfactual risks between any pair of policies can be identified.
\end{theorem}
The proof is given in
Appendix~\ref{app:additivity_implies_marginal_risk}.

Taking the expectation of the expression in Theorem~\ref{thm:additivity_marginal_risk} over the distribution of covariates, we find that the counterfactual risk is identified up to an unknown additive constant $\E_{P_{\bX}}[C_{P_{Y(\cD) \mid \bX}}(\bX)]$. We state this result as the following corollary without proof.
\begin{corollary}[The Counterfactual Risk under the Additive Loss] \label{cor:identification_risk_additive}
    Fix a joint law $P \in \cP$ and a policy $\pi: \cX \to \Delta(\cD)$. Suppose that a counterfactual loss function $\ell^\add$ takes the additive form given in Definition~\ref{def:additive}.  Then, under Assumptions~\ref{ass:consistency}--\ref{ass:bounded_loss}, the counterfactual risk can be decomposed as,
    \begin{align*}
        &R_P(\pi; \ell^\add) = \sum_{d\in \cD} \sum_{k\in \cD} \sum_{y\in \cY} \E_{P_{\bX}}\left[\omega_k(d, y, \bX) \cdot \pi(d; \bX) \cdot P(Y = y\mid D = k, \bX)\right] + \E_{P_{\bX}}[C_{P_{Y(\cD) \mid \bX}}(\bX)].
    \end{align*}
    Thus, the difference in counterfactual risk between any pair of
    decision rules is identified.
\end{corollary}

Theorem~\ref{thm:additivity_marginal_risk} immediately implies that given $P \in \cP$, one can find an optimal policy by minimizing the identifiable term alone, 
$$\begin{aligned}
\pi^\text{opt} \ \in & \ \argmin_{\pi \in \Pi} R_P(\pi; \ell^\add) \\
= & \ \argmin_{\pi \in \Pi} \sum_{d\in \cD} \sum_{k\in \cD} \sum_{y\in \cY} \E_{P_{\bX}}\left[\omega_k(d, y, \bX) \cdot \pi(d; \bX) \cdot P(Y = y\mid D = k, \bX)\right],
\end{aligned}
$$
where $\Pi$ is a class of policies to be considered.  Moreover,
if the constant term is zero, then we can exactly identify the
counterfactual risk. This happens, for example, when the intercept
function in the additive counterfactual loss
(Definition~\ref{def:additive}) is zero, i.e., $\varpi(\by, \bx)=0$
for all $\by \in \cY^\cD$ and $\bx \in \cX$. Lastly, by placing
assumptions on $\varpi$, we can control the magnitude of
$\E_{P_{\bX}}[C_{P_{Y(\cD) \mid \bX}}(\bX)]$, which leads to partial identification results.

A natural question arises as to whether or not the above
identification results can be obtained under a different form of
counterfactual loss function.  The next theorem shows that the
additivity of the loss function given in Definition~\ref{def:additive}
is a necessary and sufficient condition for the identification of
difference in counterfactual risk.  In other words, every
counterfactual risk that is identifiable up to an additive constant
must have an additive loss.
\begin{theorem}[Additivity as Necessary and Sufficient Condition for Identification]\label{thm:add_necessary_sufficient}
  Fix a counterfactual loss $\ell$.  Under
  Assumptions~\ref{ass:consistency}--\ref{ass:bounded_loss}, the risk
  difference, $R_P(\pi; \ell) - R_P(\rho; \ell)$, is identifiable in
  the sense of Definition~\ref{def:identifiability} for every pair of
  policies $\pi, \rho: \cX \to \Delta(\cD)$ if and only if the loss
  $\ell$ admits the additive representation given in
  Definition~\ref{def:additive}.
\end{theorem}
The proof is given in Appendix~\ref{app:add_necessary_sufficient} and it shows that it suffices to prove Theorem~\ref{thm:add_necessary_sufficient} for all pairs of \emph{deterministic policies} $\pi_d := d \in \cD$.

Finally, as a corollary of this identification result, we show that
the exact identification of the counterfactual risk is possible {\it
  if and only if} there exists an equivalent counterfactual loss that
sets the intercept function $\varpi$ to zero
(Definition~\ref{def:additive}).
\begin{corollary}[Necessary and Sufficient Condition for Exact Identification]\label{cor:standard}
    The counterfactual risk given in Definition~\ref{def:risk} is identifiable under Assumptions~\ref{ass:consistency}--\ref{ass:bounded_loss} for every policy $\pi:\cX \to \Delta(\cD)$ if and only if the loss admits the following form,
    \begin{align*}
      \ell^\add(d;\by, \bx) = \sum_{k\in \cD} \omega_k(d, y_k,\bx),
    \end{align*}
    for a weight function $\omega_k: \cD \times
    \cY \times \cX \rightarrow \R$. 
    The identification formula is given by the corresponding risk,
    \begin{align*}
        R_P(\pi; \ell^\add) &= \sum_{d \in \cD}
        \sum_{k \in \cD} \sum_{y\in \cY} \E_{P_{\bX}}\left[\omega_k(d, y, \bX) \cdot \pi(d; \bX) \cdot P(Y = y \mid D = k, \bX)\right].
    \end{align*}
\end{corollary}
The proof is given in Appendix~\ref{app:standard}.

Theorem~\ref{thm:add_necessary_sufficient} shows that additive counterfactual
losses represent the largest class of counterfactual losses that are
uniformly identifiable for every $P \in \cP$ under
Assumptions~\ref{ass:consistency}--\ref{ass:bounded_loss}. In particular,
separability alone is not sufficient for identification; what is
required is \emph{additive} separability in the potential outcomes. The counterfactual
risks associated with this class never depend on the joint
distribution of potential outcomes in decision-relevant
terms. Moreover, this class of counterfactual losses is consistent
with the decision-analytic framework of \citet{dawid2000causal}, since
the resulting quantity of interest depends only on the marginal
distributions of potential outcomes. 

\section{Preferences of Counterfactual Loss}\label{sec:preferences_of_counterfactual_loss}

The preceding section studied identification of counterfactual risk,
but did not address whether decision making with counterfactual losses
satisfies basic rationality requirements such as transitivity. By
contrast, standard losses are known to induce preferences that are
consistent with the von Neumann--Morgenstern (vNM) axioms
\citep{NeumannMorgenstern1944} when preferences are defined over the
space of realized outcomes $Y=Y(D)\in\cY$. In a companion paper
\citep{koch2026axiomaticfoundationdecisionscounterfactual}, we show
that counterfactual losses are also coherent in the vNM sense once
preferences are defined on the generalized space of \emph{potential
  outcomes}, where the decision maker may also care about
counterfactual outcomes $\{Y(d)\}_{d\in\cD}$. We briefly summarize the
main results here.

Let $\cZ:=\cD\times\cY^{\cD}\times\cX$ denote the potential outcome space. As in Section~\ref{subsec:counterfactual_loss_risk}, each policy $\pi$ and joint law $P \in \cP$ induce a distribution $P^\pi \in \Delta(\cZ)$. In this section only, we allow for \emph{oracle} policies $\pi:\cY^{\cD}\times\cX\to\Delta(\cD)$, so that every element of $\Delta(\cZ)$ can be represented as $P^\pi$ for some $(P, \pi)$. In particular, absent restrictions such as Assumption~\ref{ass:ignorability}, the space $\cP$ can be identified with $\Delta(\cZ)$.
Given a counterfactual loss $\ell$, define a preference relation on $\Delta(\cZ)$ by,
\[
    P^\pi \succsim Q^\rho
    \iff
    R_P(\pi;\ell) \le R_Q(\rho;\ell).
\]
This relation compares state--policy pairs across populations, and
includes as a special case comparisons of policies within a fixed
state. In \citet{koch2026axiomaticfoundationdecisionscounterfactual},
we show that these preferences are fully characterized by a
generalization of the vNM axioms to $\Delta(\cZ)$ and, in particular,
satisfy transitivity.

On the same counterfactual outcome space, we also characterize how standard and additive losses restrict preferences. Preferences induced by standard losses correspond to an additional irrelevance axiom with which the decision maker is indifferent to counterfactual outcomes. Additive losses impose a weaker restriction. They allow the decision maker to value all potential outcomes, but require indifference to their joint dependence structure.

This aligns with the identification results above. The identification problem arises because the joint distribution of potential outcomes is unobserved. Additivity is precisely the restriction that eliminates dependence on the unidentified correlation structure while still allowing preferences to depend on counterfactual outcomes.



\section{Expressiveness of Additive Counterfactual Loss}\label{sec:expressiveness_add}

Having characterized the identifiable additive class, we ask how restrictive additivity is. If an additive counterfactual loss induces the same policy ranking as a standard loss, then they are not different in a meaningful way. We show that this equivalence holds when decisions are binary, but that additive losses are generally strictly more expressive with three or more decisions. In the latter case, under binary outcomes, we interpret additive counterfactual risk in terms of both accuracy and difficulty, whereas standard risk captures only accuracy. Finally, we give a sufficient monotonicity condition under binary outcomes under which every counterfactual loss admits an additive representation.

\subsection{When Does the Additive Counterfactual Loss Differ from the Standard Loss?}\label{subsec:add_std_equiv}
We show that when the decision is binary ($K=2$), every additive counterfactual loss has an equivalent standard loss that induces the same ranking over policies and hence the same optimal decision. However, if the decisions are non-binary, additive counterfactual losses are generally strictly more expressive than standard losses, provided at least one counterfactual weight has sufficient interactions between decision and counterfactual potential
outcomes.

\begin{proposition}[Additive Counterfactual Risk with Binary Decision]\label{pro:binary_add_std_equiv}
Suppose that the decision is binary, i.e., $\cD = \{0, 1\}$. Suppose that the counterfactual loss $\ell^\add$ takes the additive form given in Definition~\ref{def:additive} and that Assumption~\ref{ass:bounded_loss} holds. Define
\begin{align*}
    \ell^\std(d; y_d, \bx) & = \omega_d(d, y_d, \bx) - \omega_d(1-d,
                             y_d, \bx), \\
    h(\by, \bx) & =  \omega_0(1, y_0, \bx) + \omega_1(0, y_1, \bx) + \varpi(\by, \bx).
\end{align*}
Then,
\begin{align*}
    \ell^\add(d; \by, \bx) = \ell^\std(d; y_d, \bx) + h(\by, \bx),
\end{align*}
for all $d \in \{0, 1\}, \by \in \cY^2$ and $\bx \in
\cX$. Consequently, for any additive counterfactual loss $\ell^\add$
there exists a standard loss $\ell^\std$ such that for every joint law
\(P\in\cP\) and every policy \(\pi:\cX\to\Delta(\cD)\), the risk
difference,
$R_P(\pi;\ell^\add) - R_P(\pi;\ell^\std) = \E_{P_{Y(\cD),\bX}}[h(Y(0),
Y(1), \bX)]$, does not depend on \(\pi\). Moreover, the standard loss
$\ell^\std$ is unique for this property up to addition of a bounded
measurable function of covariates only, i.e.,
$\ell^{\std}_c(d;y_d,\bx)=\ell^{\std}(d;y_d,\bx) + c(\bx)$. Finally,
the mapping from additive counterfactual losses to their equivalent
standard losses is surjective but not injective.
\end{proposition}

The proof is given in Appendix~\ref{app:binary_add_std_equiv}. This
proposition implies that when the decision is binary, every additive
counterfactual loss has an equivalent standard loss that induces the
same ordering over policies. As a consequence, the corresponding risks
lead to the same optimal decision rule.

Intuitively, when the decision is binary there is only one comparison:
choosing $d$ versus choosing $1-d$, that is, considering
$\ell(d; \by, \bx) - \ell(1-d; \by, \bx)$.  For an identifiable loss,
this contrast must be additive in the potential outcomes. Since $1-d$
is a deterministic function of $d$, it must be of the form
$f_1(y_1, \bx) - f_0(y_0, \bx)$. But, with only two potential
outcomes, any such additive contrast can be absorbed into a standard
loss contrast, for example by defining
$\ell^{\std}(d; y_d, \bx) = f_d(y_d, \bx)$.  On the other hand, when
$K = 3$, the relevant comparisons involve a third, unchosen
option. Because a standard loss difference depends on at most two
potential outcomes, it cannot generally capture this additional
dependence.

Perhaps surprisingly, this
result is not limited to binary outcomes. The corresponding standard
loss $\ell^\std(d; y_d, \bx)$ is given by the difference between the weight
for the realized outcome $Y(D^\pi)=y_d$ under the chosen decision
$D^\pi = d$ and the counterfactual weight for the same outcome $Y(D^\pi)=y_d$ under the
alternative decision $D^\pi = 1 - d$ (up to an additive constant that
only depends on covariates).

Although every additive counterfactual loss has a unique corresponding
standard loss (up to an additive constant) that yields the same policy ranking, the reverse does not hold.  That is, any given
standard loss has infinitely many additive counterfactual losses with
the same policy ranking. Each of such counterfactual losses
assigns different values to principal strata. This asymmetry explains
why, even in the binary case, additive counterfactual loss has a
straightforward interpretation based on principal strata.

We now revisit our binary decision examples given in Section~\ref{subsec:additivity}. 
\setcounter{example}{0}
\begin{example}[Counterfactual Classification Loss]
The counterfactual loss in Equation~\eqref{eq:classification_loss_general} can be expressed as a standard loss,  $\ell^\std(d; y_d) = \ell_{y_d} - \tilde{\ell}_{y_d} + c_d$,  
which corresponds to the difference between the realized and
counterfactual losses associated with outcome $y_d$ with the additional
cost term.  Thus, the regret is absorbed into the standard loss. If $\tilde{\ell}_0 < \tilde{\ell}_1$ (see Section~\ref{sec:add_counterfactual_loss}), the loss difference between correct and incorrect decisions is amplified, since $(\ell_0 - \tilde{\ell}_0) - (\ell_1 - \tilde{\ell}_1) > \ell_0 - \ell_1$.
\end{example}

\begin{example}[Asymmetric Counterfactual Loss]
When $\ell_0^\text{R} - \ell_1^\text{R} = \ell_0^\text{H} - \ell_1^\text{H} = \kappa$ holds, the loss in Equation~\eqref{eq:asymmetric_loss} can be rewritten as a standard loss with 
$\ell^\std(d; y_d) = \ell_{y_d} + c_d - \{\kappa - (\ell_0 -
\ell_1)\}y_{d} \propto \ell^{\text{H}}_{y_d} + c_d \propto
\ell^{\text{R}}_{y_d} + c_d$. Thus, the loss reduces to the standard loss in
Equation~\eqref{eq:ex_standard_loss}.
\end{example}

Despite Proposition~\ref{pro:binary_add_std_equiv}, additive and
standard losses can still induce different preferences on the
potential outcome space $\Delta(\cZ)$ even in the binary decision case \citep[Section~4.2]{koch2026axiomaticfoundationdecisionscounterfactual}. The
reason is that, although the risk difference
$R_P(\pi;\ell^{\add}) - R_P(\pi;\ell^{\std}) = H_P$ is policy
independent, it can vary with $P\in\cP$. The next remark discusses how
this affects common decision criteria.

\begin{remark}[Connection to Decision Criteria]
  Empirical risk minimization (ERM) is a common learning strategy
  where one chooses a policy that minimizes the estimated risk
  $\hat R(\cdot; \ell)$, leading to
  $\hat{\pi}_{\mathrm{ERM}} \in \argmin_{\pi \in \Pi} \hat R(\pi;
  \ell)$. By Theorem~\ref{thm:add_necessary_sufficient}, this approach
  requires additivity.  Alternatively, one may appeal to standard
  decision principles such as the Bayes, minimax, and minimax regret
  criteria \citep{berger1985statistical, manski2004statistical,
    STOYE200970}. Under the Bayes principle, given an assignment mechanism $e$, one specifies a prior
  $\xi$ over states of nature $P_{Y(\cD), \bX} \in \Delta(\cY^\cD \times \cX)$ and chooses
  $\pi_{\mathrm{Bayes}} \in \argmin_{\pi \in \Pi} \int
  \mathcal{R}_{e \cdot P_{Y(\cD), \bX}}(\pi; \ell)\, \mathrm{d} \xi(P_{Y(\cD), \bX})$, where $e \cdot P_{Y(\cD), \bX}\in \cP$ and $\mathcal{R}_{e \cdot P_{Y(\cD), \bX}}$ defined as in Equations~\eqref{eq:joint_law} and~\eqref{eq:sampling_risk}. Similarly, the minimax criterion
  selects the policy that minimizes worst-case risk,
  $\pi_{\mathrm{MM}} \in \argmin_{\pi \in \Pi} \sup_{P_{Y(\cD), \bX} \in \Delta(\cY^\cD \times \cX)} \mathcal{R}_{e \cdot P_{Y(\cD), \bX}}(\pi; \ell)$.  Finally, the minimax
  regret criterion chooses the policy that minimizes worst-case regret
  relative to the oracle policy,
  $$\pi_{\mathrm{MM-R}} \in \argmin_{\pi \in \Pi} \sup_{P_{Y(\cD), \bX} \in \Delta(\cY^\cD \times \cX)} \left\{ \mathcal{R}_{e \cdot P_{Y(\cD), \bX}}(\pi; \ell) - \min_{\pi' \in
      \Pi} R_{e \cdot P_{Y(\cD), \bX}}(\pi'; \ell) \right\}.$$
      
    In the binary case $\cD=\{0,1\}$, the decomposition $\ell^{\add}=\ell^{\std}+h$ in Proposition~\ref{pro:binary_add_std_equiv} with a decision independent function $h$ implies that $\ell^{\add}$ and $\ell^{\std}$ yield the same optimal policy under ERM, Bayes, and minimax regret. This invariance need not hold under minimax as the decision-independent term $h$ can shift risks differently across states $P_{Y(\cD), \bX}$.

\end{remark}

Next, we show that when the decision is non-binary, no standard loss
corresponds directly to the additive counterfactual loss, as long as
the counterfactual weight is not additively separable in decision and
counterfactual outcome.
\begin{proposition}[Additive Counterfactual Risk with Non-binary Decision]\label{prop:counterfactual_needed}
  Assume that the decision is non-binary, i.e., $K = |\cD| \ge
  3$. Suppose that the counterfactual loss $\ell^\add(d; \by, \bx)$
  takes the additive form given in Definition~\ref{def:additive} and
  that Assumption~\ref{ass:bounded_loss} holds.

  Then, there exists no standard loss $\ell^\std(d; y_d, \bx)$ such
  that for every $P \in \cP$ the risk difference
  $R_P(\pi;\ell^\add)-R_P(\pi;\ell^\std)$, is independent of
  $\pi: \cX \to \Delta(\cD)$ if and only if at least one weight
  function $\omega_k(d, y, \bx)$ is not off-diagonal additively
  separable in $d \neq k$ and $y \in \cY$.  That is, there do not
  exist bounded measurable functions
  $f_k:\cD \setminus \{k\} \times \cX \to \R$ and
  $g_k: \cY \times \cX \to \R$ such that
  $\omega_k(d, y, \bx) = f_k(d, \bx) + g_k(y, \bx)$, for all
  $d \in \cD \setminus \{k\}, y \in \cY$ and $\bx \in \cX$.

  If all the weight functions are off-diagonal additively separable,
  then an equivalent standard loss exists and is given by,
\[
\ell^\std(d; y_d, \bx) = \omega_d(d, y_d, \bx) - g_d(y_d, \bx) + \sum_{k \in \cD, k \neq d} f_k(d, \bx).
\]
This standard loss is unique up to the addition of a bounded measurable function of the covariates.
\end{proposition}
The proof is given in Appendix~\ref{app:counterfactual_needed}.
Proposition~\ref{prop:counterfactual_needed} shows that in the
non-binary setting, standard losses are inadequate when the
interactions between decision and counterfactual outcomes matter. In
the binary decision case, the weight functions are always additively
separable. Indeed, since $\cD \setminus \{k\} = \{1-k\}$ is a
singleton for $\cD = \{0, 1\}$, we can set $f_k(1-k, \bx) = 0$ and
$g_k(y, \bx) = \omega_k(1-k, y, \bx)$ to recover
Proposition~\ref{pro:binary_add_std_equiv}.

\begin{example}[Trichotomous decision]
  The additive counterfactual loss given in
  Equation~\eqref{eq:tri_additive_loss} is not off-diagonal additively
  separable. Indeed, the weight function
  $\omega_k(d,y) = \indicator{d = k} (\ell_{y} + c_d) + \indicator{d
    \neq k} r_{d,k} y \indicator{k < d}$ is not separable in the
  decision and outcome for $d\neq k$ if and only if $r_{2,1} \neq 0$
  or $r_{1,0} \neq r_{2,0}$. Thus, no standard loss can take harm to
  the arrestee into account.
\end{example}

However,
Propositions~\ref{pro:binary_add_std_equiv}~and~\ref{prop:counterfactual_needed}
also show that there do exist counterfactual losses depending on all
potential outcomes that do not generalize the decision-making compared
to losses only depending on the observed outcome. We will illustrate
this somehow paradoxical behavior with an example.

\begin{example}[Additively Separable Regret]
  Consider the following additive counterfactual loss,
  $$\ell(d; \by) = \ell_{d, y_d} + \sum_{k \in \cD, k \neq d} \lt(c_{k,
    d} + \tilde \ell_{k, y_k}\rt),$$ for $d \in \cD$ and $\by \in \cY^\cD$,
  which generalizes Equation~\eqref{eq:classification_loss_general} to
  multi-valued decision and outcome settings. The corresponding weight
  function is given by
  $\omega_k(d, y) = \indicator{d = k}\ell_{d, y}+ \indicator{d \neq
    k}\lt(c_{k, d} + \tilde \ell_{k, y}\rt)$, which is off-diagonal
  additively separable with $f_k(d) = c_{k,d}$ and
  $g_k(y) = \tilde{\ell}_{k,y}$.  By
  Proposition~\ref{prop:counterfactual_needed}, an equivalent standard
  loss is equal to, 
  $$\ell^\std(d, y_d) = \ell_{d, y_d} - \tilde \ell_{d, y_d} + \sum_{k
    \in \cD, k \neq d} c_{k, d}.$$ Computing the loss difference yields
  $\ell(d, \by) - \ell^\std(d, y_d) = \sum_{k \in \cD} \tilde \ell_{k,
    y_k}$, which does not depend on the decision.  Thus, even losses
  that assign a different regret value to each unrealized
  counterfactual outcome can be represented as a standard loss, so
  long as the regret only depends additively on the decision and
  counterfactual outcome. By
  Proposition~\ref{prop:counterfactual_needed}, all counterfactual
  losses that can be written equivalently as standard losses have this
  form.
\end{example}

\subsection{Case with Binary Outcome and Multi-valued Decision}\label{subsec:binary}

Next, we show that when the outcome is binary, additive counterfactual
risk admits a decomposition into \emph{accuracy} and \emph{difficulty}
terms. Accuracy captures how well the chosen decision performs on the
realized outcome. Difficulty measures how many counterfactual
alternatives would have produced a successful outcome. Standard losses
reflect accuracy, but cannot capture difficulty.

Consider the case of binary outcome and multi-valued decision, i.e.,
$\cY=\{0,1\}$ and $K \ge 3$. For notational simplicity, we ignore
covariates.  Without loss of generality, assume that $Y=1$ is a
desirable outcome while $Y=0$ is undesirable. Under this setting, it
is natural to consider the following ordering of weight functions in
the additive counterfactual loss in Definition~\ref{def:additive} for
any given decision $d$,
\begin{equation}\label{equ:binary_loss_structure_general}
    \omega_d(d, 1)  \leq 0 \leq  \omega_d(d, 0), \quad
    \{\omega_{d^\prime}(d, 0) \}_{d^\prime \neq d} \leq 0 \leq \{\omega_{d^\prime}(d, 1) \}_{d^\prime \neq d},
\end{equation}
where $\omega_d(d, y_d)$ for $y_d = 0,1$ accounts for the {\it
  accuracy} of decision, and
$\{\omega_{d^\prime}(d, y_{d^\prime}) \}_{d^\prime \neq d}$ for
$y_{d^\prime}=0,1$ capture its {\it difficulty}.
  
We first discuss the accuracy component. The correct ``true positive''
decision, i.e., $(D^\pi, Y(d)) = (d, 1)$ decreases the
counterfactual loss, which is represented by a negative sign of
$\omega_d(d, 1)$.  In contrast, the incorrect ``false negative''
decision, i.e., $(D^\pi, Y(d)) = (d, 0)$ increases the loss,
which corresponds to a positive sign of $\omega_d(d, 0)$.

While accuracy is an essential part of the standard decision theory,
difficulty is unique to the counterfactual loss.  Avoidance of an
incorrect decision (``true negative''), i.e.,
$(D^\pi, Y(d^\prime)) = (d, 0)$, results in a decrease of the loss,
while failure to make a correct ``true positive'' decision, i.e.,
$(D^\pi, Y(d^\prime)) = (d, 1)$, increases the loss. Thus, the
counterfactual loss is smaller when fewer alternative decisions would
yield a desirable outcome (making a correct decision is
difficult). The loss is greater when many alternatives would result in
a desirable outcome (making a correct decision is easier).  In this
way, the counterfactual loss rewards a consequential decision that
changes the outcome.

Lastly, since each weight function \( \omega_k(d, y) \) depends on
both the decision and the outcome, we can decompose it into a baseline
decision term and an interaction component:
$$
\omega_k(d, y) = \omega_k(d, 0) + \indicator{y = 1} \left( \omega_k(d, 1) - \omega_k(d, 0) \right).
$$
The first term corresponds to the baseline cost of making decision $D^\pi = d$ on the potential outcome $Y(k)$, while the second term accounts for the interaction between decision and potential outcome.

We formalize the above discussion as the following corollary of
Theorem~\ref{thm:additivity_marginal_risk}.  Specifically, we show
that the conditional counterfactual risk can be decomposed into four
components; accuracy based on the realized outcome
$P^\pi(D^\pi = d, Y(d) = 1 \mid \bX = \bx)$, difficulty based on
counterfactual outcomes $P^\pi(D^\pi = d, Y(k)=1 \mid \bX = \bx)$ for
$k\ne d$, a baseline decision probability
$P^\pi(D^\pi = d \mid \bX = \bx)$, and the intercept function
$C_{P_{Y(\cD) \mid \bX}}(\bx)$.  Each of these terms has a weight
function based on the specified additive loss function.
\begin{corollary}[Accuracy and Difficulty for the Binary Outcome Case]\label{cor:bin_standard}
    Assume $\cY = \{0, 1\}$. Suppose that the counterfactual loss $\ell^\add$ takes the additive form given in Definition~\ref{def:additive} and that Assumption~\ref{ass:bounded_loss} holds. Then for every joint law $P \in \cP$ and policy $\pi: \cX \to \Delta(\cD)$, the conditional counterfactual risk can be written as,
    \begin{align*}
         & R_{P_{Y(\cD) \mid \bX = \bx}}(\pi; \ell^\add) \\
         = & \sum_{d \in \cD} \zeta_{d}(d, \bx) P^\pi(D^\pi = d, Y(d) = 1 \mid \bX = \bx)
         + \sum_{d \in \cD} \sum_{k \in \cD, k \neq d} \zeta_{k}(d, \bx)  P^\pi(D^\pi = d, Y(k) = 1 \mid \bX = \bx) \\
         & + \sum_{d \in \cD} \xi(d, \bx)  P^\pi(D^\pi = d \mid \bX = \bx) + C_{P_{Y(\cD) \mid \bX}}(\bx).
    \end{align*}
    where
    \begin{align*}
      \zeta_{k}(d, \bx) & = \omega_k(d, 1, \bx) - \omega_k(d, 0, \bx),
                          \qquad \xi(d, \bx) = \sum_{k \in \cD}\omega_k(d, 0, \bx),\\
    C_{P_{Y(\cD) \mid \bX}}(\bx) & =  \sum_{\by \in \{0, 1\}^{K}} \varpi(\by, \bx)
    P(\bY(\cD) = \by\mid \bX = \bx).
    \end{align*} 
\end{corollary}
The proof is given in Appendix~\ref{app:bin_standard}.  If we choose
the weights such that
Equation~\eqref{equ:binary_loss_structure_general} holds, then we have
$\zeta_{d}(d, \bx) \leq 0 \leq \zeta_{k}(d, \bx)$ for all $k \neq
d$. This implies that the risk decreases with greater accuracy (when
the realized outcome is desirable), and increases with counterfactual
regret (when other decisions also would yield a desirable outcome),
reflecting the difficulty of the decision. The baseline decision term
is a regularizer that penalizes decisions with higher baseline losses.

We next characterize when the additive counterfactual loss differs
from the standard loss in the binary outcome case. This result is a
corollary of Proposition~\ref{prop:counterfactual_needed}.  It shows
that an additive counterfactual loss cannot be represented by a
standard loss if and only if the former includes a decision-dependent
difficulty term.
\begin{corollary}[Decision-Dependent Difficulty for the Binary Outcome Case]\label{cor:bin_equivalence}
Assume $\cY = \{0, 1\}$ and $K = |\cD| \geq 3$. Suppose that the counterfactual loss $\ell^\add(d; \by, \bx)$ takes the additive form given in Definition~\ref{def:additive} and that Assumption~\ref{ass:bounded_loss} holds. Then, there exists a standard loss $\ell^\std(d; y_d, \bx)$ such that for every $P \in \cP$ the risk difference
        $R_P(\pi;\ell^\add)-R_P(\pi;\ell^\std)$,
  is independent of $\pi: \cX \to \Delta(\cD)$ if and only if for each $k \in \cD$ the slope
        $\zeta_{k}(d, \bx) = \omega_k(d, 1, \bx) - \omega_k(d, 0, \bx)$
    is independent of $d \in \cD \setminus \{k\}$. Under this condition an equivalent standard loss is given by
    \[
        \ell^{\std}(d; y_d, \bx) = y_d [\zeta_{d}(d, \bx) - \tilde{\zeta}_d(\bx)] + \xi(d, \bx),
    \]
    where $\tilde{\zeta}_d(\bx):=\omega_d(j,1,\bx)-\omega_d(j,0,\bx)$ for any $j\neq d$ and
    $\xi(d, \bx) = \sum_{k \in \cD}\omega_k(d, 0, \bx)$. This standard loss is unique up to the addition of a bounded measurable function of the covariates.

\end{corollary}
Thus, a necessary and sufficient condition for an additive
counterfactual loss to differ from a standard loss is to have the
regret weight function $\zeta_{k}(d, \bx)$ in
Corollary~\ref{cor:bin_standard} depend on the decision $d$.  

\subsection{When Is Every Counterfactual Loss
  Additive?}\label{subsec:all_cf_loss_is_add}

In Section~\ref{subsec:add_std_equiv}, we discussed the conditions
under which an additive counterfactual loss is equivalent to a
standard loss. Here, we briefly consider the opposite question. When
can a fully counterfactual loss be written in an additive form? We
give a monotonicity condition under which any counterfactual loss
admits an additive representation in the important case of binary
outcomes. This extends the monotonicity result of \citet{li2019unit}
to non-binary decisions. A necessary and sufficient condition is given
later as Corollary~\ref{cor:identification_verification}.  The
following result formalizes the above discussion.
\begin{proposition}[Counterfactual Risk with Binary Outcomes under Monotonicity]\label{pro:all_cf_loss_is_add}
  Suppose that $Y\in\{0,1\}$ and the potential outcomes are monotone
  in the decision, i.e., $Y(0) \le Y(1) \le \cdots \le Y(K-1)$.  Then,
  every counterfactual loss $\ell$ in Definition~\ref{def:loss} admits
  an additive representation on the monotone support. Specifically,
  there exists an additive loss $\ell^{\add}$ according to
  Definition~\ref{def:additive} such that
  $\ell(d;\by,\bx)=\ell^{\add}(d;\by,\bx)$, for all $d \in \cD$,
  $\by \in \cY^\cD_{\mathrm{mon}}$ and $\bx \in \cX$ where
  $\cY^\cD_{\mathrm{mon}}=\{s^0,\ldots,s^K\}\subseteq\{0,1\}^{\cD}$ is
  the maximal possible support of $(Y(0),\ldots,Y(K-1))$ under
  monotonicity, with $s_k^m = \indicator{k \geq m}$ for
  $k = 0, \ldots, K-1$ and all $m = 0, \ldots, K$.  A choice of
  weights for $\ell^\add$ is given by
\[
    \omega_k(d; y, \bx) = \frac{1}{K} \ell(d; s^K, \bx) + y[\ell(d; s^k, \bx) - \ell(d; s^{k+1}, \bx)].
\]
\end{proposition}
The proof is given in Appendix~\ref{app:all_cf_loss_is_add}. The
weights $\omega_k(d;y,\bx)$ in
Proposition~\ref{pro:all_cf_loss_is_add} assign to $Y(k)$ the
incremental change in loss from setting $Y(k)=0$ to $Y(k)=1$. We
revisit the examples in Section~\ref{sec:examples} to illustrate this
result under monotonicity.

\setcounter{example}{1}
\begin{example}[Asymmetric Counterfactual Loss]
  Consider monotonicity, $Y(0)\le Y(1)$, implying that no unit is
  harmed. Then, the support is given by
  $\cY^\cD_{\mathrm{mon}}=\{(1,1),(0,1),(0,0)\}$, and the
  counterfactual loss in Equation~\eqref{eq:asymmetric_loss} is
  additive,
  $\ell^{\cf}(d;y_0,y_1)=\omega_0(d;y_0)+\omega_1(d;y_1)$. The weights
  in Proposition~\ref{pro:all_cf_loss_is_add} are given by
  $\omega_k(d;y) = \frac{1}{2}(\ell_0+c_d) +
  y\{\indicator{k=0}(\ell_1-\ell_d^{\mathrm R}) +
  \indicator{k=1}(\ell_d^{\mathrm R}-\ell_0)\}$, which is the loss
  difference between never survivor to responder ($k=1$) and between
  responder to always survivor ($k=0$).
\end{example}

\begin{example}[Trichotomous Decision]
  Recall from Section~\ref{subsec:trichotomous_decision} that
  $Y(2)= 1$. Under monotonicity, we have $Y(0)\le Y(1)$, implying
  that no unit commits a crime only under supervision. In this case,
  $\cY^\cD_{\mathrm{mon}}= \{(1, 1, 1), (0,1, 1), (0, 0, 1) \}$ and
  the counterfactual loss in
  Equation~\eqref{eq:tri_counterfactual_loss} is additive with weights
  in Proposition~\ref{pro:all_cf_loss_is_add} given by
  $\omega_k(d;y) = \frac{1}{3}(\ell_0+c_d)+y\,\alpha_{k,d}$, where
  $\alpha_{k,d} = \indicator{k=d}(\ell_1-\ell_0) +
  \indicator{d=0}\{\indicator{k=1}-\indicator{k=0}\}p_0 +
  \indicator{k<d}r_{d,k}.$
\end{example}

\section{Symbolic Verification of Additivity}\label{sec:symbolic}

We have established additivity as a necessary and sufficient condition
for identifiability. In this section, we show how to verify this
condition in practice. We formulate identification as a linear system
whose solution determines whether a proposed counterfactual loss is
additive.  We then apply the analysis to the binary decision and
outcome case where we derive the resulting identification constraints
in terms of principal strata.

\subsection{A Symbolic Algorithm for Identification Verification}

Theorem~\ref{thm:add_necessary_sufficient} establishes additivity as a
necessary and sufficient condition for identifying counterfactual risk
differences. Building on this result, we propose a symbolic algorithm
that takes a counterfactual loss function as an input and tests
whether the corresponding counterfactual risk is identifiable.  The
algorithm also yields the weight and intercept functions from
Definition~\ref{def:additive}, leading to a decomposition of
counterfactual risk into observable marginals. The proposed algorithm
further clarifies the constraints that additivity imposes on principal
strata. For notational simplicity, we suppress covariates though our
analysis applies conditional on covariates.

We begin by observing that the additive form in
Definition~\ref{def:additive} induces a finite system of linear
equations.  Let
$\bm{\ell} = \{\ell^\add(d; y_0, \ldots, y_{K-1})\}_{d \in \cD, \by
  \in \cY^\cD} \in \R^{K M^K}$ denote the vector of all loss values,
$\bw^\add = \{\omega_k(d, y), \varpi(\by)\}_{d \in \cD, k \in \cD, y
  \in \cY, \by \in \cY^{K}} \in \R^{K^2M + M^K}$ denotes the
corresponding vector of weight and intercept values. The additive
structure determines a fixed $\abs{\bm{\ell}} \times \abs{\bw}$ matrix
$\bA^\add$ of zeros and ones such that, given $\bA^\add$, a loss
vector $\bm{\ell}$ is additive if and only if the following linear
equation has a solution,
\begin{align}\label{eq:linear_system}
    \bA^\add \bw^\add = \bm{\ell}.
\end{align}
An explicit example is provided in Appendix~\ref{app:weights}. This
leads to the following corollary of
Theorem~\ref{thm:add_necessary_sufficient}.

\begin{corollary}[Identification Verification of Counterfactual
  Risk]\label{cor:identification_verification} Suppose
  Assumptions~\ref{ass:consistency}--\ref{ass:ignorability} hold. Consider a
  counterfactual loss $\ell$ defined in Definition~\ref{def:loss}. The counterfactual risk difference $R_P(\pi; \ell) - R_P(\rho; \ell)$ is identifiable for every pair of policies $\pi, \rho \in \Delta(\cD)$ if and only if the linear system in Equation~\eqref{eq:linear_system} has a solution.
\end{corollary}
If the support of $Y(\cD)$ lies in a subset $S\subseteq \cY^\cD$, as in Proposition~\ref{pro:all_cf_loss_is_add}, 
then after restricting $\bA^\add$ and $\bm{\ell}$ in Equation~\eqref{eq:linear_system} to rows with $\by\in S$, Corollary~\ref{cor:identification_verification} yields necessary and sufficient criteria for identifiability on $S$.

The linear system in Equation~\eqref{eq:linear_system} can be checked
for a solution and solved using a standard linear program solver whose
time complexity is polynomial in $M$ but exponential in $K$. Whenever
a solution exists, the linear system directly yields the weight and
intercept functions, and thus a representation of the counterfactual
risk in terms of observable marginals. The solution, however, need not
be unique.

For exact identification, we remove the intercept terms from the
additive representation. This gives the restricted system
\begin{align}\label{eq:linear_system_exact}
   \bA^\exa \bw^\exa = \bm{\ell},
\end{align}
where
$\bw^\exa = \{\omega_k(d, y)\}_{d \in \cD, k \in \cD, y \in \cY} \in
\R^{K^2M}$ contains only the weight functions, $\bA^\exa$ is adjusted
accordingly, and $\bm{\ell}$ remains $KM^K$ dimensional.  The next
corollary shows a necessary and sufficient condition for exact
identification.
\begin{corollary}[Exact-Identification Verification of the
  Counterfactual Risk]\label{cor:exact_identification_verification}
    Suppose Assumptions~\ref{ass:consistency}--\ref{ass:ignorability} hold. Consider a
  counterfactual loss $\ell$ defined in Definition~\ref{def:loss}. The counterfactual risk $R_P(\pi; \ell)$ is identifiable for every policy $\pi \in \Delta(\cD)$ if and only if the linear system in Equation~\eqref{eq:linear_system_exact} has a solution.
\end{corollary}

Lastly, by further restricting the linear system given in
Equation~\eqref{eq:linear_system} to
\begin{align}\label{eq:linear_system_std}
    \bA^\std \bw^\std = \bm{\ell},
\end{align}
where
$\bw^\std = \{\omega_d(d, y)\}_{d \in \cD,\, y \in \cY} \in \R^{KM}$
contains only the diagonal weight functions, $\bA^\std$ is defined
accordingly, and $\bm{\ell} \in \R^{KM^K}$ as before.  We obtain a
necessary and sufficient condition under which a counterfactual loss
can be equivalently written as a standard loss.

\citet{li2024unit} propose a {\it numerical} algorithm for
verification of exact identification given experimental
data. Corollary~\ref{cor:identification_verification} extends their
result to risk differences, and proposes a {\it symbolic} verification
algorithm that does not require data and operates directly on a loss
function itself.  Our algorithm also yields an explicit decomposition
of counterfactual risk into observable marginals.

\subsection{Restrictions on Principal Strata}
\label{subsec:principal_strata}

Corollaries~\ref{cor:identification_verification} and~\ref{cor:exact_identification_verification} characterize the relationship between a counterfactual loss and principal strata required for identification. Using these results, we derive the exact restrictions on loss values needed for identification in the binary decision and binary outcome case though the results can be generalized. As before, we suppress covariates, though all statements can be interpreted conditional on them. We refer to Appendix~\ref{app:weights} for details. 

Let $\mathrm{im}(\bA^\add)$ denote the
image of $\bA^\add$, i.e., the space of additive counterfactual
losses. If $\bm{\ell} \in \mathrm{im}(\bA^\add)$, then there exists a solution $\bw$ to
Equation~\eqref{eq:linear_system} and the counterfactual risk difference is identifiable.
By standard linear algebra, $\mathrm{im}(\bA^\add) = \ker\lt((\bA^\add)^\top\rt)^\perp$
\citep{Lax2007}, i.e., the orthogonal complement of the kernel of
$(\bA^\add)^\top$. Let $\bm{N}^\add$ be a matrix whose columns form a
basis for $\ker\lt((\bA^\add)^\top\rt)$. Then
$\bm{\ell} \in \mathrm{im}(\bA^\add)$ if and only if
$\lt(\bm{N}^\add\rt)^\top \bm{\ell} = \bm{0}$. This condition
specifies the strata-level restrictions on the loss needed for
identification. Similar reasoning applies to $\bA^\exa$ and
$\bA^\std$.

\begin{table}[t]
\centering
\begin{tabular}{|c|c|c|}
\hline
\multirow{2}{*}{\textbf{Principal Strata}} &  \multicolumn{2}{c|}{\textbf{Decision}} \\
\cline{2-3} & $D = 0$ & $D = 1$ \\
     \hline
     \begin{tabular}{c}
\textbf{Never Responder} \\
$(Y(0), Y(1)) = (0, 0)$
\end{tabular} & \cellcolor{orange!60} $\ell(0; 0, 0)$ & \cellcolor{blue!60}$\ell(1; 0, 0)$\\
     \hline
     \begin{tabular}{c}
\textbf{Harmed} \\
$(Y(0), Y(1)) = (1, 0)$
\end{tabular}& \cellcolor{blue!15} $\ell(0; 1, 0)$ & \cellcolor{orange!60} $\ell(1; 1, 0)$\\
     \hline
      \begin{tabular}{c}
\textbf{Responsive} \\
$(Y(0), Y(1)) = (0, 1)$
\end{tabular}& \cellcolor{blue!60} $\ell(0; 0, 1)$ & \cellcolor{orange!15} $\ell(1; 0, 1)$\\
     \hline
     \begin{tabular}{c}
\textbf{Always Responder} \\
$(Y(0), Y(1)) = (1, 1)$
\end{tabular}&  \cellcolor{orange!15} $\ell(0; 1, 1)$ & \cellcolor{blue!15} $\ell(1; 1, 1)$\\
     \hline
\end{tabular}
\caption{Counterfactual Loss for Principal Strata in the Binary Case
  $\ell(D; Y(0), Y(1))$.}  \label{tab:principal_strata}
\end{table}

Table~\ref{tab:principal_strata} shows all eight possible loss values
(two possible decisions across four principal strata), i.e.,
$\ell(d; y_0, y_1)$ for $d, y_0, y_1 \in \{0,1\}$, based on the
example given in Section~\ref{subsec:hippocratic_oath}.  Restricting to
the class of standard losses given in
Equation~\eqref{eq:linear_system_std}, the equality constraints
implied by $\lt(\bm{N}^\std\rt)^\top \bm{\ell} = \bm{0}$ are given by,
\begin{equation}
    \ell(0; 0, 0) = \ell(0; 0, 1), \quad \ell(0; 1, 0) = \ell(0; 1, 1),\quad
    \ell(1; 0, 0) = \ell(1; 1, 0), \quad \ell(1; 0, 1) = \ell(1; 1, 1). \label{eq:standard}
\end{equation}
  
These represent the familiar restrictions from standard statistical
decision theory, where the loss depends only on the chosen decision
and its realized outcome. Indeed, when $D=0$, never-responders and
responders, for whom we have $Y(0)=0$, share the same loss. Similarly,
always-responders and harmed units, for whom we have $Y(0)=1$, have the
same loss. When $D=1$, never-responders and harmed units yield the
same loss, as do always-responders and responders. Together, this
imposes four equality constraints. In
Table~\ref{tab:principal_strata}, these correspond to equality between
the light orange and purple cells within each column, and separately
between the dark orange and purple cells.  In this case, only four of
the twelve possible weight functions are non-zero:
\begin{align*}
    \omega_0(0,0) = \ell(0; 0, y_1), \quad \omega_0(0,1) = \ell(0;1,y_1),\quad
    \omega_1(1,0) = \ell(1;y_0,0), \quad \omega_1(1,1) = \ell(1;y_0,1),
\end{align*}
for $y_0, y_1 \in \{0,1\}$.

We can further relax these conditions by considering the exact
identification case in Equation~\eqref{eq:linear_system_exact}, where
$\lt(\bm{N}^\exa\rt)^\top \bm{\ell} = \bm{0}$ yields,
\begin{align}
    \ell(d; 0, 0) + \ell(d; 1, 1) = \ell(d; 0, 1) +\ell(d; 1, 0), \label{eq:exact}
\end{align}
where for each decision $d$, the sum of losses for never- and always-responders equals that for the harmed and responsive units. For example, 
Equation~\eqref{eq:exact} implies that, for each fixed decision $d$,
the decision-maker assigns the same risk to a population composed of
$50\%$ always-responders and $50\%$ never-responders as to a population
composed of $50\%$ responders and $50\%$ harmed units. In
Table~\ref{tab:principal_strata}, this corresponds to the orange cells
summing to the purple cells within each column. As mentioned earlier,
the restriction is identical to the \emph{gain equality} of
\citet{li2019unit}.  Given these two equality restrictions, if we let
$\omega_0(1,1) = a$ and $\omega_1(0,1) = b$, we can express the
remaining six weights as follows,
\begin{align*}
    \omega_0(0,0) = \ell(0; 0,1) - b, \quad \omega_0(0,1) = \ell(0; 1,1) - b, &\quad
    \omega_1(1,0) = \ell(1; 1,0) - a, \quad \omega_1(1,1) = \ell(1; 1,1) - a,
    \\
    \omega_0(1,0) = \ell(1; 0,1) - \ell(1; 1,1) + a, &\quad \omega_1(0,0) = \ell(0; 1,0) - \ell(0; 1,1) + b.
\end{align*}

We can further relax the equality constraints to achieve
identification only up to a constant by considering
Equation~\eqref{eq:linear_system}. In this case,
$\lt(\bm{N}^\add\rt)^\top \bm{\ell} = \bm{0}$ imposes a single
restriction,
\begin{equation}
    [\ell(1; 0, 0)-\ell(0; 0, 0)] + [\ell(1; 1, 1) - \ell(0; 1, 1)]  = [\ell(1; 1, 0) - \ell(0; 1, 0)] + [\ell(1; 0, 1) - \ell(0; 0, 1)], \label{eq:constant}
\end{equation}
which represents the sum of loss differences between the two decisions
for never-responders and always-responders is equal to the corresponding
sum for the harmed and responsive units. 
For example, equation~\eqref{eq:constant} implies that the
decision-maker makes the same treatment recommendation for a
population composed of $50\%$ always-responders and $50\%$ never-responders as for a population composed of $50\%$ responders and $50\%$
harmed units.
In Table~\ref{tab:principal_strata}, this restriction corresponds to the
sum of orange cells equaling the sum of purple cells across both
columns. Solving the system under this restriction yields the weights,
\begin{align*}
    \omega_0(0,0) = \ell(0; 0,1) - b - c, &\quad \omega_0(0,1) = \ell(0; 1,1) - b - e, \\
\omega_1(1,0) = \ell(1; 1,0) - a - d, &\quad \omega_1(1,1) = \ell(1; 1,1) - a - e, \\
\omega_0(1,0) = \ell(1; 0,1) - \ell(1; 1,1) + a - c + e , &\quad \omega_1(0,0) = \ell(0; 1,0) - \ell(0; 1,1) + b - d + e , \\
\varpi(0,0) = \ell(1; 0,0) - \ell(1; 1,0) - &\ell(1; 0,1) + \ell(1; 1,1) + c + d - e,
\end{align*}
where $\omega_0(1,1) = a, \omega_1(0,1) = b, \varpi(0,1) = c,
\varpi(1,0) = d, \varpi(1,1) = e$ are free parameters.

\begin{table}[t]
\centering
\begin{tabular}{|w{c}{2.75in}|w{c}{0.8in}|w{c}{1.2in}|}
\hline
\multirow{2}{*}{\textbf{Restrictions on loss functions}}   & \textbf{Degrees of} & \multirow{2}{*}{\textbf{Identification}}  \\
& \textbf{freedom} & 
\\ 
\hline
\begin{tabular}{c}
light (dark) orange = light (dark) purple \\
within each column (Equation~\eqref{eq:standard})
\end{tabular}
& 4 &\cellcolor{green!30} Exact 
\\
\hline
\begin{tabular}{c}
sum of orange $=$ sum of purple \\
within each column (Equation~\eqref{eq:exact}) \\
\end{tabular} & 6 &\cellcolor{green!30} Exact* 
\\
\hline
\begin{tabular}{c}
sum of orange = sum of purple \\
across two columns (Equation~\eqref{eq:constant})
\end{tabular} & 7 &\cellcolor{olive!30} Constant* 
\\
\hline
No restriction & 8 & \cellcolor{red!30} Unidentifiable 
\\
\hline
\end{tabular}
\caption{Summary of identifiability restrictions referring to Table~\ref{tab:principal_strata}   under which the risk
  is either exactly identifiable (``Exact'' in green), identifiable up
  to a constant (``Constant'' in yellow), or unidentifiable (in
  red). Strong ignorability is assumed. An asterisk $*$ signifies a necessary and sufficient condition. The second column shows the number of free parameters in the loss specification.}
\label{tab:identifiability_binary_case}
\end{table}

The third column of Table~\ref{tab:identifiability_binary_case}
summarizes the identification results discussed in this subsection.
The table illustrates how the identification conditions translate into
explicit equality constraints on the loss values across principal strata.

\section{Concluding Remarks}

In this paper, we generalize the classical statistical decision theory
for treatment choice by introducing counterfactual losses that
depend on all potential outcomes. While standard losses evaluate
decisions based solely on realized outcomes, counterfactual losses
additionally take into account what would have happened under
alternative decisions. Counterfactual losses can account
for ethical and other considerations that are not possible to express
with standard losses.

While counterfactual losses are much more expressive than the
standard losses, a central challenge is the difficulty of
identification. Since only one potential outcome is observed per unit,
counterfactual risks are generally unidentifiable. We show that under
the standard assumption of strong ignorability, a counterfactual loss
leads to an identifiable risk if and only if it is additive in the
potential outcomes. Moreover, we establish that such additive
counterfactual losses do not have an equivalent standard loss
representation so long as decisions are non-binary. Finally, our proposed 
symbolic linear inverse formulation provides a practical way to verify whether a given
counterfactual loss is additive.

Our results should allow practitioners to consider a wider class of
loss functions when evaluating decision rules and learning optimal
policies.  Promising directions for future research include extending
the framework to continuous decisions and outcomes, relaxing the
strong ignorability assumption in the identification analysis
\citep[see][]{ben-michaelDoesAIHelp2024} and studying the sampling
risk in Equation~\eqref{eq:sampling_risk} when policies are learned
from finite data under counterfactual losses.

\newpage
\bibliography{ref}

\newpage
\appendix
\renewcommand\thefigure{\thesection.\arabic{figure}}    
\setcounter{figure}{0}
\numberwithin{figure}{section}
\renewcommand\thetable{\thesection.\arabic{table}}    
\setcounter{table}{0} 
\numberwithin{table}{section}
\renewcommand\theequation{\thesection.\arabic{equation}}    
\setcounter{equation}{0}
\numberwithin{equation}{section}
\renewcommand\theproposition{\thesection.\arabic{proposition}}    
\setcounter{proposition}{0}
\numberwithin{proposition}{section}

\begin{center}
 \huge Supplementary Appendix 
\end{center}

\section{Numerical Illustrations for Section~\ref{sec:examples}}
\label{app:numerical_illustrations}
\subsection{Binary Decisions}\label{app:numerical_binary}

We provide a numerical example for Section~\ref{subsec:hippocratic_oath} to build intuition. Normalize the cost of no treatment to $c_0=0$ and set the treatment cost to $c_1=0.1$. Set the counterfactual losses for forgone death and survival to $\tilde \ell_0=0.25$ and $\tilde \ell_1=0.5$, so that $\tilde \ell_0<\tilde \ell_1$ encodes regret from failing to choose a life-saving option. Suppose that $45\%$ of the population are responders and $30\%$ are harmed. Always survivors and never survivors do not affect the induced decision rules in this setup, so we leave them unspecified.

Figure~\ref{fig:decision_boundaries_binary} displays the resulting decision recommendations under the three losses. The decision boundaries under the standard and additive losses are parallel. The additive loss shifts the boundary by $\tilde \ell_0-\tilde \ell_1$, illustrating that a shifted standard loss can replicate the behavior. By contrast, the fully counterfactual loss yields a boundary that is not parallel. The dotted $45^\circ$ line is the identification line, along which the loss admits an additive (and hence identifiable) representation.

\begin{figure}[htbp]
  \centering
  \includegraphics[width=0.9\linewidth,keepaspectratio]{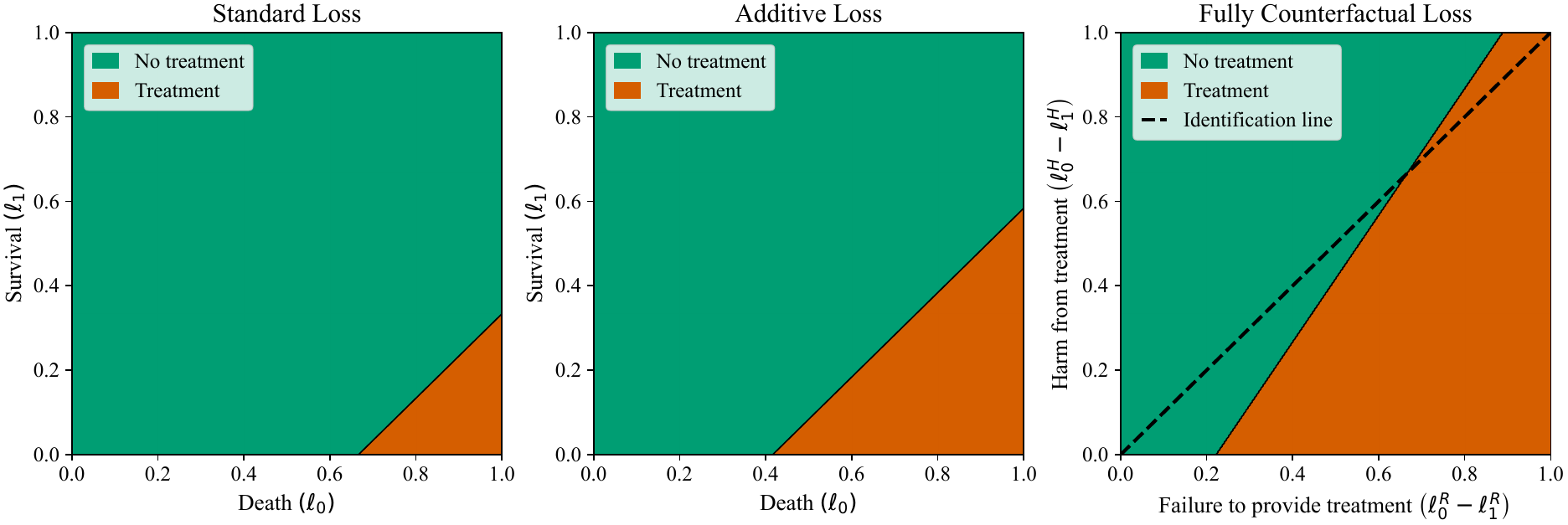}
  \caption{Decision recommendations in the binary medical example of Section~\ref{app:numerical_binary} under the standard, additive counterfactual, and fully counterfactual loss.
  The standard and additive losses have parallel decision boundaries, whereas the fully counterfactual loss does not. The dotted line in the fully counterfactual panel marks the identification region for that loss.}
  \label{fig:decision_boundaries_binary}
\end{figure}

\subsection{Trichotomous Decisions}\label{app:numerical_trichotomous}

We provide a numerical example for Section~\ref{subsec:trichotomous_decision} to build intuition. The crime losses are normalized as $\ell_0 = 1$
and $\ell_1 = 0$.  Suppose that the cost of release on their own
recognizance is negligible ($c_0 = 0$), whereas release with
supervision incurs a small loss ($c_1 = 0.03$).  Detention is assumed
to be the most costly ($c_2 = 0.3$). We set the regret of unnecessary
detention to $r_{20} = r_{21} = r$, and that of unnecessary
supervision to $r_{10} = r/10$.  Lastly, we assume a symmetric penalty
$p_0 = p_1 = p$ so that harm of release to the public equals harm of
supervised release to the public.

Since detention leads to no crime ($Y(2) = 1$), the population
consists of four principal strata: those who never commit crime
($(Y(0), Y(1), Y(2)) = (1, 1, 1)$), those who commit crime only if
released on their own recognizance ($(Y(0), Y(1), Y(2)) = (0, 1, 1)$),
those who commit crime only if released under supervision
($(Y(0), Y(1), Y(2)) = (1, 0, 1)$), and those who commit crime unless detained
($(Y(0), Y(1), Y(2)) = (0, 0, 1)$). For this example, we set the
population shares of the above principal strata to 20\%, 25\%, 20\%, and
35\%, respectively.

\begin{figure}[t]
  \centering
  \includegraphics[width=0.9\linewidth,keepaspectratio]{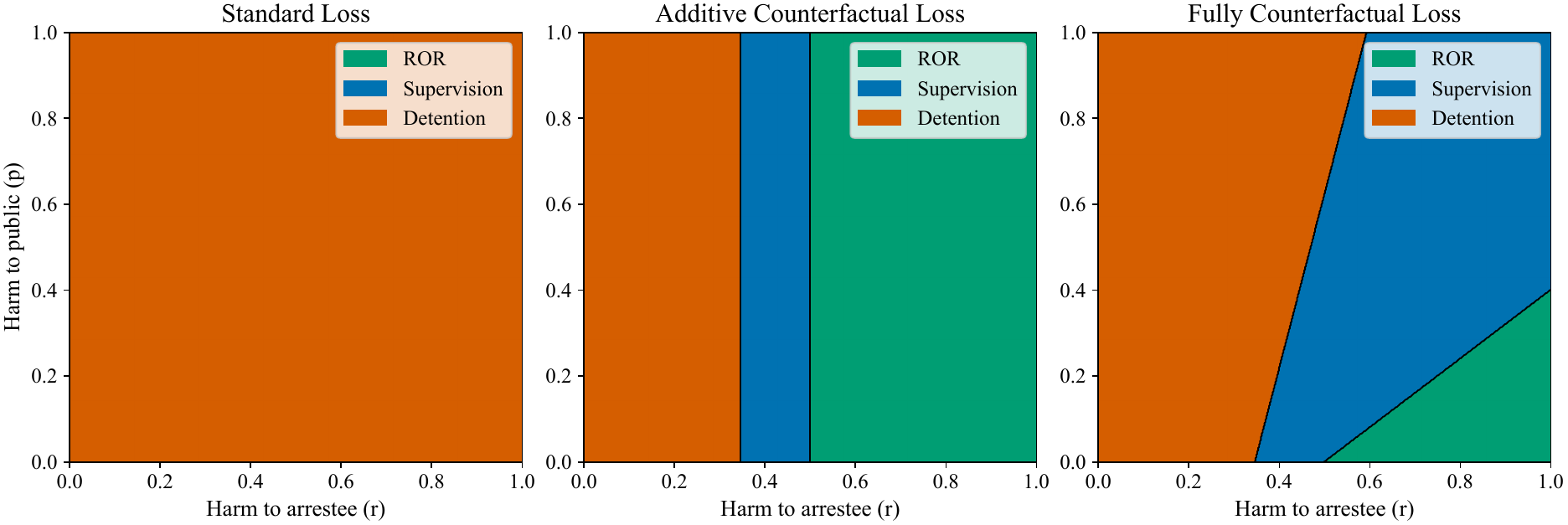}
  \caption{
  Decision recommendations in the trichotomous criminal justice example of Section~\ref{app:numerical_trichotomous} under the standard, additive counterfactual, and fully counterfactual loss. The standard loss gives the same decision
    recommendation throughout, as it ignores harm to the public $p$ and harm to
    the arrestee $r$. The additive counterfactual loss varies with $r$, whereas the fully counterfactual loss varies with both types of harm.}
  \label{fig:decision_boundaries}
\end{figure}

Figure~\ref{fig:decision_boundaries} shows the decision
recommendations under the three loss functions considered in this
subsection. Under the standard loss, detention is always optimal as
both harms are ignored. The additive loss incorporates harm to the
arrestee, leading to release on their own recognizance when $r$ is
sufficiently large. The fully counterfactual loss accounts for both
types of harm, and recommends alternative decisions depending on their
relative magnitude.  The example illustrates the three loss functions
can yield different decision recommendations.  In comparison to the
standard loss, the counterfactual loss functions incorporate the
notions of harm and yield a richer analysis. 

\section{Mathematical Proofs}

\subsection{Proof of Theorem~\ref{thm:additivity_marginal_risk}}
\label{app:additivity_implies_marginal_risk}

By the definition of the counterfactual conditional risk (Definition~\ref{def:risk}), we have,
\begin{align}
    & R_{P_{Y(\cD) \mid \bX = \bx}}(\pi; \ell^\add)\nonumber \\
    = \ & \sum_{d\in \cD} \sum_{\by \in \cY^\cD} \lt(\varpi(\by, \bx) + \sum_{k\in \cD} \omega_k(d, y_k,\bx)\rt) \nonumber \\
    & \hspace{1in} \cdot \pi(d; \bx) \cdot P(Y(0) = y_0,\ldots, Y(K-1)=y_{K-1}\mid \bX = \bx) \nonumber\\
    = \ & \sum_{d\in \cD} \sum_{\by \in \cY^\cD} \sum_{k\in \cD} \omega_k(d, y_k,\bx)  \cdot \pi(d; \bx) \cdot P(Y(0) = y_0,\ldots, Y(K-1)=y_{K-1}\mid \bX = \bx) \nonumber\\
    &\quad + \sum_{d\in \cD} \sum_{\by \in \cY^\cD} \varpi(\by, \bx) \cdot \pi(d; \bx) \cdot P(Y(0) = y_0,\ldots, Y(K-1)=y_{K-1}\mid \bX = \bx), \label{eq:risk_decomp}
\end{align}
where the first equality follows from Definition~\ref{def:additive}.  We now consider each of the two summation terms in Equation~\eqref{eq:risk_decomp}.  For notational convenience, we write $\by_{-k} = (y_0, \ldots, y_{k-1}, y_{k+1}, \ldots, y_{K-1})$ for any $k \in \cD$. We can rewrite the first term as,
\begin{align*}
    &  \sum_{d\in \cD} \sum_{\by \in \cY^\cD} \sum_{k\in \cD} \omega_k(d, y_k,\bx)  \cdot \pi(d; \bx) \cdot P(Y(0) = y_0,\ldots, Y(K-1)=y_{K-1}\mid \bX = \bx) \\
    = \ & \sum_{d \in \cD} \sum_{k\in \cD} \sum_{y_k \in \cY} \sum_{\by_{-k} \in \cY^{K-1}} \omega_k(d, y_k,\bx) \cdot \pi(d; \bx) \cdot P(Y(0) = y_0,\ldots, Y(K-1)=y_{K-1}\mid \bX = \bx) \\
    = \ & \sum_{d \in \cD} \sum_{k \in \cD} \sum_{y_k \in \cY} \omega_k(d, y_k,\bx) \cdot \pi(d; \bx) \sum_{\by_{-k} \in \cY^{K-1}}  P(Y(0) = y_0,\ldots, Y(K-1)=y_{K-1}\mid \bX = \bx) \\
    = \ & \sum_{d \in \cD} \sum_{k \in \cD} \sum_{y_k\in \cY} \omega_k(d, y_k,\bx) \cdot \pi(d; \bx)  \cdot P(Y(k) = y_k\mid \bX = \bx) \\
    = \ & \sum_{d \in \cD} \sum_{k \in \cD} \sum_{y \in \cY} \omega_k(d, y,\bx) \cdot \pi(d; \bx)  \cdot P(Y(k) = y \mid \bX = \bx) \\
    = \ & \sum_{d \in \cD} \sum_{k \in \cD} \sum_{y \in \cY} \omega_k(d, y,\bx) \cdot \pi(d; \bx)  \cdot P(Y = y \mid D = k, \bX = \bx),
\end{align*}
where the last equality follows from Assumptions~\ref{ass:consistency} and~\ref{ass:ignorability}.
    
We next rewrite the second summation in
Equation~\eqref{eq:risk_decomp} as follows,
\begin{align*}
    &\sum_{d\in \cD} \sum_{\by \in \cY^\cD}  \varpi(\by, \bx) \cdot \pi(d; \bx)  \cdot P(Y(0) = y_0,\ldots, Y(K-1)=y_{K-1} \mid \bX = \bx) \\
    = \ & \sum_{\by \in \cY^\cD}  \varpi(\by, \bx) P(Y(0) = y_0,\ldots, Y(K-1)=y_{K-1} \mid \bX = \bx) \sum_{d\in \cD} \pi(d; \bx) \\
    = \ & \sum_{\by \in \cY^\cD}  \varpi(\by, \bx) P(Y(0) = y_0,\ldots, Y(K-1)=y_{K-1}\mid \bX = \bx).
    \end{align*}
    This completes the proof.
\qed

\subsection{Proof of Theorem~\ref{thm:add_necessary_sufficient}}
\label{app:add_necessary_sufficient}
First suppose that $\ell$ is additive according to
Definition~\ref{def:additive}. Then, under
Assumptions~\ref{ass:consistency}--\ref{ass:bounded_loss}, by
Theorem~\ref{thm:additivity_marginal_risk}, for a given $P \in \cP$
and policy $\pi: \cX \to \Delta(\cD)$, the conditional counterfactual
risk takes the following form,
\[
    R_{P_{Y(\cD) \mid \bX = \bx}}(\pi; \ell) = \sum_{d\in \cD} \sum_{k\in \cD} \sum_{y\in \cY} \omega_k(d, y, \bx) \cdot \pi(d; \bx) \cdot P(Y = y\mid D = k, \bX = \bx) + C_{P_{Y(\cD) \mid \bX}}(\bx),
\]
where $C_{P_{Y(\cD) \mid \bX}}(\bx)$ does not depend on the policy
$\pi$. Therefore, for policies $\pi$ and $\rho$, the conditional risk
difference, which is given by,
\begin{align*}
     & R_{P_{Y(\cD) \mid \bX = \bx}}(\pi; \ell) -  R_{P_{Y(\cD) \mid \bX
         = \bx}}(\rho; \ell) \\
     =  & \sum_{d\in \cD} \sum_{k\in \cD} \sum_{y\in \cY} \omega_k(d, y, \bx) \cdot \{\pi(d; \bx) - \rho(d; \bx) \} \cdot P(Y = y\mid D = k, \bX = \bx),
\end{align*}
depends only on the observed conditional law $Q_P(D,Y\mid\bX)$.  This
is because,
\[
   P(Y = y\mid D = k, \bX = \bx) = \frac{Q_P(D = k, Y = y \mid \bX = \bx)}{Q_P(D = k \mid \bX = \bx)} \qquad P_{\bX}\text{-a.s.}
\]
Integrating over $P_{\bX}$ identifies
$R_{P}(\pi; \ell) - R_{P}(\rho; \ell)$ according to
Definition~\ref{def:identifiability} by the subsequent discussion in
Section~\ref{sec:nonparametric_identification}.

Now, suppose that for every pair of policies $(\pi, \rho)$, the risk
difference $R_{P}(\pi; \ell) - R_{P}(\rho; \ell)$ is identifiable
uniformly across $P \in \cP$ according to
Definition~\ref{def:identifiability}. Fix $d_0 \in \cD$ and let
$\pi_d := d$ be the constant policy that always chooses $d$. By
assumption, $R_{P}(\pi_d; \ell) - R_{P}(\pi_{d_0}; \ell)$ is
identified for every $d \in \cD$.

Fix $\bx \in \cX$. Associate $\Delta(\cY^\cD)$ with the simplex
$\Delta_{M^K - 1} = \{\bp \in \R^{M^K} : \sum_{\by \in \cY^\cD}
p_{\by} = 1 \text{ and } p_{\by} \geq 0\}$. Let
$\bp, \bp^\prime \in \Delta_{M^K - 1}$ with the same marginals, i.e.,
$\bC \bp = \bC \bp^\prime$, where $\bC$ is the marginalization matrix
of Lemma~\ref{lem:marginalization}.  Define
$P_{Y(\cD) \mid \bX},P_{Y(\cD) \mid \bX}^\prime \in \Delta(\cY^\cD
\times \cX)$ as,
\[
    P_{Y(\cD), \bX}(Y(\cD) = \by, \bX \in A) = p_{\by} \cdot \delta_{\bx}(A), \qquad P_{Y(\cD), \bX}^\prime(Y(\cD) = \by, \bX \in A) = p^\prime_{\by} \cdot \delta_{\bx}(A),
\]
where $\delta_{\bx}(\cdot)$ is the covariate Dirac measure at $\bX =
\bx$. Given a valid assignment mechanism $e$ (e.g., choose $e(d; \bx)
= 1/K$), let $P = e \cdot P_{Y(\cD), \bX}$ and $P^\prime = e \cdot
P_{Y(\cD), \bX}^\prime$ defined as in
Equation~\eqref{eq:joint_law}. By construction, we have $P, P^\prime
\in \cP$.  Then,
\begin{align*}
    Q_P(D = k, Y = y, \bX \in A) &= \sum_{\by \in \cY^{\cD}} \indicator{y_k = y} P(D = k, Y(\cD) = \by, \bX \in A)
    \\
    &= \sum_{\by \in \cY^{\cD}} \indicator{y_k = y} \cdot \int_{A} e(k; \bx) \cdot P_{Y(\cD) \mid \bX}(Y(\cD) = \by \mid \bX = \bx) dP_{\bX}(\bX=\bx),
    \\
    &= \sum_{\by \in \cY^{\cD}} \indicator{y_k = y} \cdot e(k; \bx) \cdot p_{\by} \cdot \delta_{\bx}(A)
    \\
    &= e(k; \bx) \cdot \delta_{\bx}(A) \cdot \sum_{\by \in \cY^{\cD}} \indicator{y_k = y} \cdot p_{\by}
    \\
    &= e(k; \bx) \cdot \delta_{\bx}(A) \cdot \sum_{\by \in \cY^{\cD}} \indicator{y_k = y} \cdot p^\prime_{\by}
    \\
    &= Q_{P^\prime}(D = k, Y = y, \bX \in A),
\end{align*}
implying that the observed laws, $P$ and $P^\prime$, coincide.
Definition~\ref{def:identifiability} implies
\begin{align*}
    R_{P}(\pi_d; \ell) -  R_{P}(\pi_{d_0}; \ell) = R_{P^\prime}(\pi_d; \ell) -  R_{P^\prime}(\pi_{d_0}; \ell).
\end{align*}
Since both states put mass on $\bX = \bx$, this is equivalent to
\[
    \sum_{\by \in \cY^\cD} \{\ell(d; \by, \bx) - \ell(d_0; \by, \bx)\} p_{\by} = \sum_{\by \in \cY^\cD} \{\ell(d; \by, \bx) - \ell(d_0; \by, \bx)\} p^\prime_{\by},
\]
where the statement of
Lemma~\ref{lem:additivity}\ref{lem:additivity_b} is applied to
$a_{\by}(d, \bx) := \ell(d; \by, \bx) - \ell(d_0; \by, \bx)$. By
Lemma~\ref{lem:additivity}\ref{lem:additivity_d}, there exist
$\omega_{k, y}(d, \bx) \in \R$ and $k \in \cD, y \in \cY$ such that
the following equality holds for all $\by \in \cY^\cD$,
\begin{align}\label{eq:thm2_add_l}
  \ell(d; \by, \bx) - \ell(d_0; \by, \bx) = \sum_{k \in \cD} \omega_{k, y_k}(d, \bx).
\end{align}
Since $\cD$ and $\cY$ are finite, $\{\omega_{k, y}(d, \bx)\}_{d, k \in \cD, y \in \cY}$ can be chosen as bounded measurable functions of $\bx$.
Define $\varpi(\by, \bx) := \ell(d_0; \by, \bx)$ and let
\begin{align*}
\omega_k(d, y, \bx) :=
\begin{cases}
0, & d = d_0, \\
\omega_{k,y}(d, \bx), & d \neq d_0 .
\end{cases}
\end{align*}
Then, Equation~\eqref{eq:thm2_add_l} yields,
\[
     \ell(d; \by, \bx) = \varpi(\by, \bx) + \sum_{k \in \cD} \omega_k(d, y_k, \bx).
\]
Since $\bx \in \cX$ was arbitrary, $\ell$ is additive by
Definition~\ref{def:additive}. This completes the proof.

\qed

\subsection{Proof of Corollary~\ref{cor:standard}}
\label{app:standard}

The result follows immediately from the proof of Theorem~\ref{thm:add_necessary_sufficient} by considering the risk instead of the risk difference.\qed

\subsection{Proof of Proposition~\ref{pro:binary_add_std_equiv}}\label{app:binary_add_std_equiv}

We first show the existence of such a standard loss by proving the decomposition 
\[
    \ell^\add(d; \by, \bx) = \ell^\std(d; y_d, \bx) + h(\by, \bx),
\]
where $\ell^\std$ is the equivalent standard loss and $h(\by, \bx)$ is
the decision independent term given in
Proposition~\ref{pro:binary_add_std_equiv}. By
Definition~\ref{def:additive}, direct calculation yields,
\begin{align*}
    &\ell^\add(d; \by, \bx) - \lt[\ell^\std(d; y_d, \bx) + h(\by, \bx)\rt]
    \\
    = \ &\lt[\omega_d(d, y_d, \bx) + \omega_{1-d}(d, y_{1-d}, \bx) + \varpi(\by, \bx) \rt]
    \\
    &\quad - \lt[\omega_d(d, y_d, \bx) - \omega_d(1-d, y_d, \bx) + \omega_0(1, y_0, \bx) + \omega_1(0, y_1, \bx) + \varpi(\by, \bx)\rt]
    \\
     = \ & \lt[\omega_{1-d}(d, y_{1-d}, \bx) + \omega_d(1-d, y_d, \bx)\rt] - \lt[\omega_0(1, y_0, \bx) + \omega_1(0, y_1, \bx)\rt]
    \\
    = \ & 0.
\end{align*}
This proves the decomposition.

Next, we show that $\ell^\std$ induces the same ranking over policies
as $\ell^\add$.  Fix $P \in \cP$ and policy
$\pi: \cX \to \Delta(\cD)$, then by Definition~\ref{def:risk}, we have,
\begin{align*}
    R_P(\pi; \ell^\add) &= \E_{P^\pi}[\ell^\add(D^\pi; Y(0), Y(1), \bX)]
    \\
    &= \E_{P^\pi}[\ell^\std(D^\pi; Y(D^\pi), \bX)] + \E_{P_{Y(\cD), \bX}}[h(Y(0), Y(1), \bX)]
    \\
    &= R_P(\pi; \ell^\std) + \E_{P_{Y(\cD), \bX}}[h(Y(0), Y(1), \bX)].
\end{align*}
Note that
$\E_{P^\pi}[h(Y(0), Y(1), \bX)] = \E_{P}[h(Y(0), Y(1), \bX)] =
\E_{P_{Y(\cD), \bX}}[h(Y(0), Y(1), \bX)]$ as $P$ and $P^\pi$ have the
same marginal law for $(Y(0), Y(1), \bX)$.  Since the second term is
independent of $\pi$, this proves that the risk difference does not
depend on the policy.

For uniqueness, suppose that for another bounded standard loss
$\tilde \ell^\std$, the risk difference
$R_P(\pi; \ell^\add) - R_P(\pi; \tilde \ell^\std)$ is independent of
$\pi$ for every $P \in \cP$. This implies that for a constant policy
$\pi_d := d \in \cD$, we have,
\begin{align*}
    & \left[R_P(\pi_d; \ell^\add) - R_P(\pi_d; \tilde \ell^\std)\right] - \left[R_P(\pi_d; \ell^\add) - R_P(\pi_d; \ell^\std)\right] \\
    = \ & R_P(\pi_d; \ell^\std) - R_P(\pi_d; \tilde \ell^\std)
    \\
    = \ & \E_P[\ell^\std(d, Y(d), \bX) - \tilde \ell^\std(d, Y(d), \bX)] 
\end{align*}
does not depend on $d \in \{0, 1\}$ for every $P$. Hence, the
following equality holds uniformly over all $P\in\cP$,
\begin{align}\label{eq:prop1_equality}
    \E_P[\ell^\std(0, Y(0), \bX) - \tilde \ell^\std(0, Y(0), \bX)] = \E_P[\ell^\std(1, Y(1), \bX) - \tilde \ell^\std(1, Y(1), \bX)].
\end{align}
In particular, the equality holds under any degenerate law that
concentrates on a single value of $(Y(0),Y(1),\bX)$. Fix $\bx\in\cX$
and $y,y'\in\cY$, set $e(d;\bx)=1/2$ for $d\in\{0,1\}$, and define, 
\[
    P(D = d, Y(0) = y_0, Y(1) = y_1, \bX \in A) = \frac{1}{2} \indicator{y_0 = y, y_1 = y^\prime} \delta_{\bx}(A),
\]
where $A \subseteq \cX$ and $\delta_{\bx}(\cdot)$ is the Dirac measure at $\bX = \bx$.
Then $P \in \cP$. Substituting this $P$ into
Equation~\eqref{eq:prop1_equality} yields, 
\begin{align}\label{eq:prop1_equality2}
    \ell^\std(0, y, \bx) - \tilde \ell^\std(0, y, \bx) = \ell^\std(1, y^\prime, \bx) - \tilde \ell^\std(1, y^\prime, \bx).
\end{align}
Taking $y'=y$ shows that $\ell^\std(d,y,\bx)-\tilde\ell^\std(d,y,\bx)$
does not depend on $d$. Plugging this back into
Equation~\eqref{eq:prop1_equality} and varying $y,y'$ implies that the
same difference cannot depend on $y$ either. Hence,
$\ell^\std(d,y,\bx)-\tilde\ell^\std(d,y,\bx)$ is constant in $(d,y)$.
Therefore, there exists a bounded, measurable function $c: \cX \to \R$
such that
$\ell^\std(d; y, \bx) - \tilde \ell^\std (d; y, \bx)= c(\bx)$, which
proves uniqueness.

Finally, the surjectivity of the map from an additive counterfactual
loss to an equivalent standard loss follows immediately from the fact
that any standard loss is a special case of an additive counterfactual
loss. However, the map is not injective. Given any standard loss
$\ell^{\std}(d; y_d, \bx)$, fix a bounded measurable function
$h:\cY^2 \times \cX \to \R$ and define
$\ell^\add(d; \by, \bx) = \ell^\std(d; y_d, \bx) + h(\by, \bx)$. Then,
by the first part of this proposition, this corresponds to an additive
counterfactual loss with the same policy ordering. Since $h$ was
arbitrary, there exist infinitely many additive counterfactual losses
that map to the same standard loss.  \qed

\subsection{Proof of Proposition~\ref{prop:counterfactual_needed}}\label{app:counterfactual_needed}

We will prove the equivalent negated statement.  Let $\ell^\add$
denote an additive counterfactual loss defined in
Definition~\ref{def:additive}. Then, there exists an equivalent
standard loss $\ell^\std$ such that for every $P \in \cP$ the risk
difference, $R_P(\pi;\ell^\add)-R_P(\pi;\ell^\std)$, is independent of
$\pi$ if and only if all weights are off-diagonal additively separable
for some bounded, measurable functions $f_k, g_k$.  In that case, the
equivalent standard loss is given by,
\begin{align}\label{eq:prop2_equiv_std_loss}
    \ell^\std(d; y_d, \bx) = \omega_d(d, y_d, \bx) - g_d(y_d, \bx) + \sum_{k \in \cD, k \neq d} f_k(d, \bx).
\end{align}

Assume first that all weights are off-diagonal additively separable. That is, for every $k \in \cD$ there exist bounded, measurable functions $f_k:\cD \setminus \{k\} \times \cX \to \R$ and $g_k: \cY \times \cX \to \R$ such that
\[
    \omega_k(d, y, \bx) = f_k(d, \bx) + g_k(y, \bx), \qquad d \in \cD \setminus \{k\}.
\]
Then,
\[
    \ell^\add(d; \by, \bx) = \varpi(\by, \bx) + \omega_d(d, y_d, \bx) + \sum_{k \in \cD, k \neq d} \{f_k(d, \bx) + g_k(y_k, \bx)\}.
\]
Subtracting $\ell^\std$ defined in Equation~\eqref{eq:prop2_equiv_std_loss} yields
\begin{align*}
     \ell^\add(d; \by, \bx) -  \ell^\std(d; y_d, \bx) = \varpi(\by, \bx) + \sum_{k \in \cD} g_k(y_k, \bx) = h(\by, \bx),
\end{align*}
which does not depend on the decision. Taking expectations gives
\[
    R_P(\pi; \ell^\add) - R_P(\pi; \ell^\std) = \E_P[h(Y(0), \ldots, Y(K-1), \bX)],
\]
which is independent of $\pi$. Thus, $\ell^\std$ is equivalent. 

Suppose next that an equivalent standard loss $\ell^\std$ exists. For fixed $P \in \cP$, define
\[
    \phi_P(\pi) := R_P(\pi; \ell^\add) - R_P(\pi; \ell^\std).
\]
By equivalence, $\phi_P(\pi)$ is independent of $\pi$. Let $\pi_d := d \in \cD$ denote the constant decision rule. Then, 
\[
    \phi_P(\pi_d) = \lt\{\E_P[\varpi(Y(0), \ldots, Y(K-1), \bX)] + \sum_{k \in \cD} \E_P[\omega_k(d, Y(k), \bX)]\rt\} - \E_P[\ell^\std(d; Y(d), \bX)],
\]
does not depend on $d \in \cD$. This implies that $\phi_P(\pi_d) - \phi_P(\pi_{d^\prime}) = 0$ for $d, d^\prime \in \cD$, or equivalently,
\begin{align}\label{eq:omega-equality}
    \sum_{k \in \cD} \E_P[\omega_k(d, Y(k), \bX) - \omega_k(d^\prime, Y(k), \bX)]  = \E_P[\ell^\std(d; Y(d), \bX) - \ell^\std(d^\prime; Y(d^\prime), \bX)].
\end{align}

Note that the left-hand side is a function of the joint vector of potential outcomes $(Y(0), \ldots, Y(K-1))$, while the right-hand side is only a function of $(Y(d), Y(d^\prime))$, which is a true subset when $K \geq 3$. Next, fix $k \in \cD$. Since $K \geq 3$ the set $\cD \setminus \{k\}$ contains at least two decisions $d, d^\prime$. Fix $\bx \in \cX$ and let $\by, \by^\prime \in \cY^\cD$ only differ in their $k$-th coordinate, i.e., $y_j = y^\prime_j$ for $j \neq k$ and $y_k \neq y^\prime_k$. Set $e(d; \bx) = 1/K$ for $d \in \cD$ and define the joint laws,
\begin{align*}
    P(D = d, Y(\cD) & = \widetilde \by, \bX \in A) = \frac{1}{K}
                      \indicator{\widetilde \by = \by}
                      \delta_{\bx}(A), \\
  P^\prime(D = d, Y(\cD) & = \widetilde \by, \bX \in A) = \frac{1}{K} \indicator{\widetilde \by = \by^\prime} \delta_{\bx}(A),
\end{align*}
where $A \subseteq \cX$ and $\delta_{\bx}(\cdot)$ is the Dirac measure
at $\bX = \bx$. Then $P, P^\prime \in \cP$. Substituting $P, P^\prime$
into Equation~\eqref{eq:omega-equality} and recalling that $k \notin
\{d, d^\prime\}$ yield,
\begin{align*}
    0 &= [\phi_P(\pi_d) - \phi_P(\pi_{d^\prime})]- [\phi_{P^\prime}(\pi_d) - \phi_{P^\prime}(\pi_{d^\prime})] 
    \\
    &= \lt[\sum_{j \in \cD} \omega_j(d, y_j, \bx) - \omega_j(d^\prime, y_j, \bx)\rt] - \lt[\sum_{j \in \cD} \omega_j(d, y_j^\prime, \bx) - \omega_j(d^\prime, y_j^\prime, \bx)\rt]
    \\
    &= \lt[\omega_k(d,y_k,\bx)-\omega_k(d^\prime,y_k,\bx)\rt] - \lt[\omega_k(d,y^\prime_k,\bx)-\omega_k(d^\prime,y^\prime_k,\bx)\rt]
\end{align*}
which implies
$\omega_k(d,y_k,\bx)-\omega_k(d^\prime,y_k,\bx) =
\omega_k(d,y^\prime_k,\bx)-\omega_k(d^\prime,y^\prime_k,\bx)$.  Since
$y_k, y_k^\prime$ were chosen arbitrarily,
$\omega_k(d, y, \bx) - \omega_k(d^\prime, y, \bx)$ does not depend on
$y \in \cY$. Since $k$, $d, d^\prime$, and $\bx$ were arbitrary,
Lemma~\ref{lem:off_diagonal_separability} implies that all
off-diagonal weights are additively separable.

Finally, uniqueness of the equivalent standard loss follows from the same argument as in the proof of Proposition~\ref{pro:binary_add_std_equiv}.
\qed

\subsection{Proof of Proposition~\ref{pro:all_cf_loss_is_add}}
\label{app:all_cf_loss_is_add}

Fix $d \in \cD, m \in \{0, \ldots, K\}$ and $\bx \in \cX$. Define $\ell^\add$ as the additive loss as in Definition~\ref{def:additive} given by the weights in Proposition~\ref{pro:all_cf_loss_is_add} without intercept. Recall that $s^m = \{\indicator{k \geq m}\}_{k \in \cD} \in \cY^\cD_{\mathrm{mon}}$ such that
\begin{align*}
    \ell^{\add}(d;s^m,\bx) &= \sum_{k \in \cD}  \omega_k(d; s_k^m, \bx)
    \\
    &= \sum_{k \in \cD} \lt\{\frac{1}{K} \ell(d; s^K, \bx) + s_k^m [\ell(d; s^k, \bx) - \ell(d; s^{k+1}, \bx)]\rt\}
    \\
    &= \ell(d; s^K, \bx) + \sum_{k = m}^{K-1} \{\ell(d; s^k, \bx) - \ell(d; s^{k+1}, \bx)\}
    \\
    &= \ell(d; s^K, \bx) + \ell(d; s^m, \bx) - \ell(d; s^K, \bx)
    \\
    &= \ell(d; s^m, \bx).
\end{align*}
Since under monotonicity $(Y(0), \ldots, Y(K-1))$ is contained in $\cY^\cD_{\mathrm{mon}} = \{s^0,\ldots,s^K\}$, the claim follows.

\subsection{Proof of Corollary~\ref{cor:bin_standard}}
\label{app:bin_standard}

By Theorem~\ref{thm:additivity_marginal_risk}, the conditional counterfactual risk can be written as
\begin{align}
    R_{P_{Y(\cD) \mid \bX = \bx}}(\pi; \ell^\add) &= \sum_{d\in \cD} \sum_{k\in \cD} \sum_{y = 0}^1 \omega_k(d, y, \bx) \cdot \pi(d; \bx) \cdot P(Y(k) = y\mid \bX = \bx) + C_P(\bx)
    \notag
    \\
    &= \sum_{d\in \cD} \sum_{k\in \cD} \sum_{y = 0}^1 \omega_k(d, y, \bx) \cdot P^\pi(D^\pi = d, Y(k) = y\mid \bX = \bx) + C_P(\bx), \label{eq:2conditional}
\end{align}
because $P^\pi(D^\pi = d, Y(k) = y\mid \bX = \bx) = \pi(d; \bx) \cdot P(Y(k) = y\mid \bX = \bx)$.

Note that
\begin{align*}
    P^\pi(D^\pi = d, Y(k) = 0\mid \bX = \bx) =  P^\pi(D^\pi = d\mid \bX = \bx) -  P^\pi(D^\pi = d, Y(k) = 1\mid \bX = \bx).
\end{align*}
Substituting this expression into the first term of Equation~\eqref{eq:2conditional} yields,
\begin{align*}
    &\sum_{d \in \cD} \sum_{k\in\cD} \sum_{y=0}^1 \omega_k(d, y, \bx) P^\pi(D^\pi = d, Y(k) = y\mid \bX = \bx) \\
   = \ & \sum_{d \in \cD} \sum_{k\in\cD} \omega_k(d, 0, \bx) P^\pi(D^\pi = d, Y(k) = 0\mid \bX = \bx) + \omega_k(d, 1, \bx) P^\pi(D^\pi = d, Y(k) = 1\mid \bX = \bx) \\
    = \ & \sum_{d \in \cD} \sum_{k\in\cD} \Big\{ \omega_k(d, 0, \bx) \lt[P^\pi(D^\pi = d\mid \bX = \bx) - P^\pi(D^\pi = d, Y(k) = 1\mid \bX = \bx)\rt] \\
    & \quad + \omega_k(d, 1, \bx) P^\pi(D^\pi = d, Y(k) = 1\mid\bX = \bx) \Big\}
    \\
   = \ & \sum_{d \in \cD} \sum_{k\in\cD} \left\{\omega_k(d, 0, \bx) P^\pi(D^\pi = d\mid \bX = \bx) + [\omega_k(d, 1, \bx) - \omega_k(d, 0, \bx)] P^\pi(D^\pi = d, Y(k) = 1\mid \bX = \bx)\right\} \\
    = \ & \sum_{d \in \cD} \xi(d, \bx) P^\pi(D^\pi = d\mid \bX = \bx) + \sum_{d \in \cD} \sum_{k\in\cD} \zeta_k(d, \bx) P^\pi(D^\pi = d, Y(k) = 1\mid \bX = \bx) \\
    = \ & \sum_{d \in \cD} \xi(d, \bx) P^\pi(D^\pi = d\mid \bX = \bx) + \sum_{d\in \cD} \sum_{k\in \cD, k\neq d} \zeta_{k}(d, \bx) P^\pi(D^\pi = d, Y(k) = 1\mid \bX = \bx) \\
    &+ \sum_{d \in \cD} \zeta_{d}(d, \bx) P^\pi(D^\pi = d, Y(d) = 1\mid \bX = \bx),
\end{align*}
where $\zeta_{k}(d, \bx) = \omega_k(d, 1, \bx) - \omega_k(d, 0, \bx)$, $\xi(d, \bx) = \sum_{k \in \cD}\omega_k(d, 0, \bx)$. Combining this with Equation~\eqref{eq:2conditional} proves the claim.
\qed

\subsection{Lemmas and their Proofs}
\label{app:lemmas}

\subsubsection{Lemma~\ref{lem:equivalence_of_identifiability}}
\label{app:equivalence_of_identifiability}

\begin{lemma}\label{lem:equivalence_of_identifiability}
Let $\cP$ denote the set of all joint probability distributions under consideration, and let \( \cQ \) be the set of all observable distributions. Let \( Q: \cP \to \mathcal{Q} \) denote a surjective mapping that associates each joint distribution with an observable distribution. Let \( \theta(P) \) denote the (causal) estimand of interest under distribution \( P \in \cP \). For a fixed observable distribution \( Q_0 \in \mathcal{Q} \) define 
$\mathcal{S}_0 = \{ P \in \cP : Q(P) = Q_0 \}.$
Then, the following statements are equivalent:
\begin{enumerate}
\item[\textnormal{(a)}] For every $Q_0\in\mathcal{Q}$ there exists a constant
      $\theta_0$ such that $\theta(P)=\theta_0$ for all $P\in\mathcal{S}_{0}$.
\item[\textnormal{(b)}] There exists a function $h$ such that $\theta(P)  = h(Q(P))$ for every $P \in \cP$.
\end{enumerate}
\end{lemma}

\begin{proof}
We first prove that (b) implies (a). Fix any $Q_0\in\mathcal{Q}$ and let $P \in\mathcal{S}_{0}$ such that $Q(P)=Q_0$. Note that $\mathcal{S}_{0}$ is non-empty by $Q$ surjective. Then, by (b), there exists $h$ such that
\begin{align*}
    \theta(P) = h(Q(P)) = h(Q_0) = \theta_0,
\end{align*}
for every $P \in\mathcal{S}_{0}$. Because $Q_0$ was arbitrary, the claim follows.

We next prove that (a) implies (b). Let $Q_0\in\mathcal{Q}$. Surjectivity of $Q$ ensures $\mathcal{S}_{0}$ is non-empty.  By (a), there is a constant $\theta_0$ such that $\theta(P)=\theta_0$ for all $P\in\mathcal{S}_{0}$.  Define $h:\mathcal{Q}\to\R$ by setting $h(Q_0) := \theta_0$ for every such $Q_0 \in \mathcal{Q}$. Now, let $P \in \cP$ be an arbitrary probability measure and let $Q_0 = Q(P)$. Then, there exists $\theta_0 = h(Q_0)$ such that $\theta(P) = \theta_0 = h(Q_0) = h(Q(P))$. This proves the claim.
\end{proof}

\subsubsection{Lemma~\ref{lem:risk_identifiable_iff_conditional_risk_identifiable}}
\label{app:risk_identifiable_iff_conditional_risk_identifiable}

\begin{lemma}\label{lem:risk_identifiable_iff_conditional_risk_identifiable}
Fix a counterfactual loss $\ell$ satisfying Assumption~\ref{ass:bounded_loss} and a policy $\pi: \cX \to \Delta(\cD)$. Then the following are equivalent:
\begin{enumerate}
\item[\textnormal{(a)}] The counterfactual risk $R_P(\pi; \ell)$ is identifiable from $Q_P$ according to Definition~\ref{def:identifiability}.
\item[\textnormal{(b)}] The conditional counterfactual risk $R_{P_{Y(\cD) \mid \bX}}(\pi; \ell)$ is identifiable from $Q_P(D, Y \mid \bX)$, i.e., for every pair $P_1, P_2 \in \cP$ with common covariate marginal $P_{1, \bX} = P_{2, \bX} =: \mu$, whenever $Q_{P_1}(D, Y \mid \bX) = Q_{P_2}(D, Y \mid \bX)$ $\mu$-a.s., we also have $R_{P_{1, {Y(\cD) \mid \bX}}}(\pi, \ell) = R_{P_{2, {Y(\cD) \mid \bX}}}(\pi, \ell)$ $\mu$-a.s.
\end{enumerate}
\end{lemma}
\begin{proof}
We first prove that (a) implies (b). We use proof by contradiction. Suppose that $R_{P_{Y(\cD) \mid \bX}}(\pi, \ell)$ is not identifiable from $Q_P(D, Y \mid \bX)$. Then there exist two joint laws $P_1, P_2 \in \cP$ with common covariate marginal $P_{1, \bX} = P_{2, \bX} =: \mu$ such that the event
\[
    E:=\lt\{\bx \in \cX: Q_{P_1}(D, Y \mid \bX= \bx) = Q_{P_2}(D, Y \mid \bX= \bx)\rt\},
\]
has $\mu(E) = 1$, and the event
\[
A:= E \cap \lt\{\bx \in \cX: R_{P_{1, {Y(\cD) \mid \bX = \bx}}}(\pi; \ell) - R_{P_{2, {Y(\cD) \mid \bX =\bx}}}(\pi; \ell) \neq 0 \rt\}
\]
has $\mu(A) > 0$. Note that $A$ is measurable, because $\phi(\bx) := R_{P_{1, {Y(\cD) \mid \bX = \bx}}}(\pi; \ell) - R_{P_{2, {Y(\cD) \mid \bX = \bx}}}(\pi; \ell)$ is measurable as the difference of two measurable functions.

Let $A_+ = \{\bx \in A: \phi(\bx) > 0\}$ and $A_- = \{\bx \in A: \phi(\bx) < 0\}$, then $A = A_+ \cup A_-$. Since $\mu(A) > 0$ at least one of $A_+, A_-$ has positive $\mu$-measure. Without loss of generality assume $\mu(A_+) >0$ and define the covariate law $\nu(B) := \frac{\mu(B \cap A_+)}{\mu(A_+)}$ for measurable $B \subseteq \cX$. Then $\nu \ll \mu, \nu(A_+) = 1$ and $\phi(\bx) > 0$ $\nu$-a.s.

Therefore, since $\ell$ is bounded by assumption, 

\[
    \int_{\cX} \phi(\bx) d\nu(\bx)  > 0.
\]
Define the joint probability measures $P^\nu_i$ for $i = 1, 2$,
\[
    P^\nu_i(D, Y(\cD), \bX) = P_i(D, Y(\cD) \mid \bX) \cdot\nu(\bX).
\]
Since $P_{i, \bX}^\nu = \nu \ll \mu = P_{i, \bX}$, $P_i^\nu$ satisfies Assumption~\ref{ass:ignorability} and thus $P_i^\nu \in \cP$ for $i = 1, 2$.
Moreover, since the marginal distribution of $\bX$ does not affect the conditional risk, we have
\[
    R_{P_{i, {Y(\cD) \mid \bX}}}(\pi; \ell) =  R_{P^\nu_{i, {Y(\cD) \mid \bX}}}(\pi; \ell) \qquad \nu\text{-a.s.}
\]
for $i = 1, 2$. Hence
\begin{align}\label{eq:lem_identifiability}
    R_{P^\nu_1}(\pi, \ell) - R_{P^\nu_2}(\pi, \ell) &= \int_{\bX} R_{P^\nu_{1, {Y(\cD) \mid \bX = \bx}}}(\pi; \ell) -  R_{P^\nu_{2, {Y(\cD) \mid \bX = \bx}}}(\pi; \ell) d\nu(\bx)
    \notag
    \\
    &= \int_{\bX} R_{P_{1, {Y(\cD) \mid \bX = \bx}}}(\pi; \ell) -  R_{P_{2, {Y(\cD) \mid \bX = \bx}}}(\pi; \ell) d\nu(\bx)
    \notag
    \\
    &= \int_{\bX} \phi(\bx) d\nu(\bx) 
    \notag
    \\
    &> 0.
\end{align}
On the other hand, by construction, 
\[
    Q_{P^\nu_1}(D, Y, \bX)=Q_{P_1}(D, Y \mid \bX) \cdot \nu(\bX) = Q_{P_2}(D, Y \mid \bX) \cdot \nu(\bX) = Q_{P^\nu_2}(D, Y, \bX).
\]
Since $R_P(\pi, \ell)$ is identifiable and $Q_{P^\nu_1} = Q_{P^\nu_2}$, Definition~\ref{def:identifiability} implies that $R_{P^\nu_1}(\pi, \ell)$ = $R_{P^\nu_2}(\pi, \ell)$, contradicting Equation~\eqref{eq:lem_identifiability}. Thus, $R_{P_{Y(\cD) \mid \bX}}(\pi, \ell)$ must be identifiable from $Q_{P}(D, Y \mid \bX)$.

We next prove that (b) implies (a). Let $P_1, P_2 \in \cP$ with $Q_{P_1} = Q_{P_2}$. This implies that $P_{1, \bX} = P_{2, \bX} =: \mu$ and 
\[
    Q_{P_1}(D, Y \mid \bX= \bx) = Q_{P_2}(D, Y \mid \bX= \bx) \qquad \mu\text{-a.s.}
\]
Since the conditional risk is identifiable, this implies that
\[
    R_{P_{1, {Y(\cD) \mid \bX}}}(\pi; \ell) =  R_{P_{2, {Y(\cD) \mid \bX}}}(\pi; \ell) \qquad \mu\text{-a.s.}
\]
Therefore
\begin{align*}
    R_{P_1}(\pi, \ell) &= \int_{\bX} R_{P_{1, {Y(\cD) \mid \bX = \bx}}}(\pi; \ell) d\mu(\bx)
    \notag
    \\
    &=  \int_{\bX} R_{P_{2, {Y(\cD) \mid \bX = \bx}}}(\pi; \ell) d\mu(\bx)
    \notag
    \\
    &= R_{P_2}(\pi, \ell),
\end{align*}
which proves (a) by Definition~\ref{def:identifiability}.
\end{proof}

\subsubsection{Lemma~\ref{lem:marginalization}}
\begin{lemma}\label{lem:marginalization}
Associate the space of all joint laws $\Delta(\cY^{\cD})$ of $Y(\cD)  := (Y(0),\ldots,Y(K-1))$ with the simplex $\Delta_{M^K - 1} = \{\bp \in \R^{M^K} :  \sum_{\by \in \cY^\cD} p_{\by} = 1 \text{ and } p_{\by} \geq 0\}$. For $\bp \in \Delta_{M^K - 1}$ define the vector of marginal probabilities,
\[
    \bq(\bp) = \lt\{ \sum_{\by_{-k} \in \cY^{K-1}} p_{y, \by_{-k}} \rt\}_{k \in \cD, y \in \cY},
\]
where $\by_{-k}=(y_0,\ldots,y_{k-1},y_{k+1},\ldots,y_{K-1})$, so that $\bq(\bp)$ collects $\{P(Y(k)=y)\}_{k\in\cD,\ y\in\cY}$ from $\bp = \{P(Y(\cD) = \by)\}_{\by \in \cY^\cD}$

Define the matrix $\bC \in \{0, 1\}^{KM \times M^K}$ with rows indexed by $i(k, y)$ and columns indexed by $j(\by)$ defined via
\[
    C_{i(k, y), j(\by')} = \indicator{y'_k = y}.
\]
Then $\bq(\bp) = \bC \bp$ for all $\bp \in \Delta_{M^K - 1}$.
Moreover,
\[
     \mathrm{ker}(\bC) = \lt\{\bv \in \R^{M^K}: \sum_{\by_{-k}\in\cY^{K-1}}
      v_ {y,\by_{-k}}=0
      \ \text{for all } k\in\cD, y \in\cY \rt\}.
\]
In particular $\mathbf{1}^\top \bv = 0$ for all $\bv \in \mathrm{ker}(\bC)$ and $\bC$ is surjective from $\Delta(\cY^\cD) $ onto $(\Delta(\cY))^\cD$.
\end{lemma}

\begin{proof}
    For every joint probability distribution $P(Y(\cD))$, we can express its marginals $\{P(Y(k) = y)\}_{k \in \cD, y \in \cY}$ as,
    \begin{align}\label{eq:row_dy}
         P(Y(k) = y) = \sum_{\by^\prime\in \cY^\cD} \indicator{y^\prime_k = y} P(Y(\cD) = \by^\prime).
    \end{align}
    Define the matrix $\bC \in \{0, 1\}^{KM \times M^K}$ with rows indexed by $i(k, y)$ and columns index by $j(\by)$ via
    \begin{align*}
        C_{i(k, y), j(\by')} = \indicator{y'_k = y}.
    \end{align*}
    Then, by associating $\bp \in \Delta_{M^K-1}$ with $\{P(Y(\cD) = \by)\}_{\by \in \cY^\cD}$ and $\bq = \bq(\bp)$ with its marginal $\{P(Y(k)=y)\}_{k\in\cD,\ y\in\cY}$,  Equation~\eqref{eq:row_dy} can be written as,
    \[
        (\bC \bp)_{k, y} = \sum_{\by^\prime\in \cY^\cD} \indicator{y^\prime_k = y} p_{\by^\prime} = \sum_{\by_{-k} \in \cY^{K-1}} p_{y, \by_{-k}} = q_{k, y},
    \]
    implying $\bq(\bp)=\bC\bp$. Since for any collection of marginal distributions
    \(\{q_{k,y}\}_{y\in\cY}\in\Delta(\cY)\), \(k\in\cD\), the product law
    \[
        p_{\by}=\prod_{k\in\cD}q_{k,y_k}
    \]
    has these marginals, the restriction of \(\bC\) to
    \(\Delta(\cY^\cD)\) is surjective onto \((\Delta(\cY))^\cD\).

    We proceed to examine the kernel of the matrix $\bC$. Let $\bv = (v_{\by})_{\by \in \cY^\cD} \in \R^{M^K}$ then
    \[
        (\bC \bv)_{k, y} = \sum_{\by^\prime\in \cY^\cD} \indicator{y^\prime_k = y} v_{\by^\prime} = \sum_{\by_{-k} \in \cY^{K-1}} v_{y, \by_{-k}}.
    \]
    Therefore $\bC \bv = 0$ if and only if all marginals vanish, which proves the claim. Finally, if $\bv \in \mathrm{ker}(\bC)$, then for any $k \in \cD$
    \begin{align*}
    \mathbf{1}^\top\bv = \sum_{\by \in \cY^\cD} v_{\by} = \sum_{y \in \cY}  \sum_{\by_{-k}\in\cY^{K-1}} v_{y, \by_{-k}} = 0.
\end{align*}
\end{proof}

\subsubsection{Lemma~\ref{lem:additivity}}

\begin{lemma}\label{lem:additivity}
Associate the space of all joint laws $\Delta(\cY^{\cD})$ of $Y(\cD)  := (Y(0),\ldots,Y(K-1))$ with the simplex $\Delta_{M^K - 1} = \{\bp \in \R^{M^K} :  \sum_{\by \in \cY^\cD} p_{\by} = 1 \text{ and } p_{\by} \geq 0\}$. Let $\bC$ be the marginalization matrix from Lemma~\ref{lem:marginalization} and define $\bC(\Delta_{M^K -1}) = \{\bq \in \R^{KM}: \exists \bp \in \Delta_{M^K-1} \text{ such that } \bq = \bC\bp \}$, the set of marginal vectors induced by joint laws on $\cY^\cD$. Let $\ba = (a_{\by})_{\by \in \cY^\cD} \in \R^{M^K}$. Then the following are equivalent: 
\begin{enumerate}[label=\textnormal{(\alph*)}, ref=\textnormal{(\alph*)}]
    \item There exists a well-defined function $f_{\ba}: \bC(\Delta_{M^K -1}) \to \R$ such that $f_{\ba}(\bq) = f_{\ba}(\bq(\bp)) = f_{\ba}(\bC \bp) = \ba^\top \bp$ for every $\bp \in \Delta_{M^K - 1}$.
    \item \label{lem:additivity_b} For all $\bp, \bp^\prime \in \Delta_{M^K - 1}$ with $\bC \bp = \bC \bp^\prime$ we have $\ba^\top \bp = \ba^\top \bp^\prime$.
    \item $\ba \in \mathrm{im}(\bC^\top)$.
    \item \label{lem:additivity_d} There exist $w_{k, y} \in \R$, $k \in \cD, y \in \cY$ such that $a_{\by} = \sum_{k \in \cD} w_{k, y_k}$ for every $\by \in \cY^\cD$.
\end{enumerate}
\end{lemma}

\begin{proof}
We first show that (a) is equivalent to (b). Assume (a) holds for a function $f_{\ba}$ and let $\bp, \bp^\prime \in \Delta_{M^K - 1}$ with $\bC \bp = \bC \bp^\prime$. Then
\[
    \ba^\top \bp = f_{\ba}(\bC \bp) = f_{\ba}(\bC \bp^\prime) = \ba^\top \bp^\prime,
\]
which proves (b). Conversely, let (b) hold then the function $f_{\ba}: \bC(\Delta_{M^K -1}) \to \R$ with $f_{\ba}(\bC \bp) = \ba^\top \bp$ is well-defined as all $\bp, \bp^\prime \in \Delta_{M^K - 1}$ with $\bC \bp = \bC \bp^\prime$ map to the same value. This proves (a). 

Next we show that (b) implies (c). Assume (c) does not hold  such that $\ba \notin \mathrm{im}(\bC^\top)$. By standard linear algebra \citep[Chapter 3, Theorem 5]{Lax2007}, $\mathrm{im}(\bC^\top) = \mathrm{ker}(\bC)^{\perp}$. This implies that $\exists \bv = (v_{\by})_{\by \in \cY^\cD} \in \text{ker}(\bC) \setminus \{\bm{0}\}: \ba^{\top}\bv\neq 0$. 
Let $\bp = (p_{\by})_{\by \in \cY^\cD}\in \mathring{\Delta}_{M^K - 1}$ be in the interior of the simplex with strictly positive entries. Define
\begin{align*}
    \bp^\prime = \bp + \alpha \bv, \quad \text{with} \ 0 < \alpha < \frac{\min_{\by \in \cY^\cD} p_{\by}}{\max_{\by \in \cY^\cD} \abs{v_{\by}}}.
\end{align*}
Note that such an $\alpha$ exists. Indeed, by assumption, $0 < p_{\by}$ for all entries and $\bv$ has at least one non-zero entry as $\bv\neq \bm{0}$. Observe,
\begin{align*}
    \min_{\by \in \cY^\cD} p^\prime_{\by} \geq \min_{\by \in \cY^\cD} p_{\by} - \alpha \max_{\by \in \cY^\cD} \abs{v_{\by}} > 0,
\end{align*}
and,
\begin{align*}
    \sum_{\by \in \cY^\cD} p^\prime_{\by} = \sum_{\by \in \cY^\cD} p_{\by} + \alpha \sum_{\by \in \cY^\cD} v_{\by} = \sum_{\by \in \cY^\cD} p_{\by} = 1,
\end{align*}
where the second equality follows from $\mathbf{1}^\top \bv = 0$ by $\bv  \in \text{ker}(\bC)$ and Lemma~\ref{lem:marginalization}. Thus, $\bp^\prime \in \Delta_{M^K - 1}$ represents a law in $\Delta(\cY^\cD)$.

Because $\bv \in \text{ker}(\bC)$,
\begin{align*}
    \bC\bp^\prime = \bC\bigl(\bp+\alpha\bv\bigr)=\bC\bp.
\end{align*}
By (b) this implies $\ba^\top \bp = \ba^\top \bp^\prime$ and thus,
\begin{align*}
    0 = \ba^\top \bp^\prime - \ba^\top \bp = \ba^\top (\bp^\prime - \bp) = \alpha \ba^\top \bv \neq 0,
\end{align*}
as $\alpha \neq 0$ by construction and $\ba^\top \bv \neq 0$ by assumption. This yields a contradiction and proves (c).

Next we show that (c) implies (d). If $\ba \in \mathrm{im}(\bC^\top)$ then $a = \bC^\top \bw$ for some $\bw = (w_{k, y})_{k \in \cD, y \in \cY} \in \R^{KM}$. By Definition of $\bC$ in Lemma~\ref{lem:marginalization} this implies,
\[
    a_{\by} = (\bC^\top \bw)_{\by} = \sum_{k \in \cD} \sum_{y^\prime \in \cY} \indicator{y_k = y^\prime }w_{k, y^\prime} = \sum_{k \in \cD} w_{k, y_k},
\]
which proves (d).

Finally, we show that (d) implies (b). Let $\bp, \bp^\prime \in \Delta_{M^K - 1}$ with $\bC \bp = \bC \bp^\prime$. By (d),
\begin{align*}
    \ba^\top \bp &= \sum_{\by \in \cY^\cD} a_{\by} p_{\by} 
    \\
    &= \sum_{\by \in \cY^\cD} \lt( \sum_{k \in \cD} w_{k, y_k} \rt) p_{\by} 
    \\
    &= \sum_{k \in \cD} \sum_{\by \in \cY^\cD} w_{k, y_k} p_{\by}
    \\
    &= \sum_{k \in \cD} \sum_{y \in \cY} w_{k, y} \sum_{\by_{-k} \in \cY^{K-1}} p_{y, \by_{-k}}
    \\
    &= \sum_{k \in \cD} \sum_{y \in \cY} w_{k, y} (\bC \bp)_{k, y}
    \\
    &= \sum_{k \in \cD} \sum_{y \in \cY} w_{k, y} (\bC \bp^\prime)_{k, y}
    \\
    &= \sum_{k \in \cD} \sum_{y \in \cY} w_{k, y} \sum_{\by_{-k} \in \cY^{K-1}} p^\prime_{y, \by_{-k}}
    \\
    &= \sum_{k \in \cD} \sum_{\by \in \cY^\cD} w_{k, y_k} p^\prime_{\by} 
    \\
    &= \sum_{\by \in \cY^\cD} \lt( \sum_{k \in \cD} w_{k, y_k} \rt) p^\prime_{\by} 
    \\
    &= \sum_{\by \in \cY^\cD} a_{\by} p^\prime_{\by} 
    \\
    &= \ba^\top \bp^\prime,
\end{align*}
which proves (b) and completes the proof.
\end{proof}

\subsubsection{Lemma~\ref{lem:off_diagonal_separability}}
\begin{lemma}\label{lem:off_diagonal_separability}
Let $\abs{\cD} \geq 3$ and fix $k \in \cD, \bx \in \cX$. Consider the weight function $\omega_k: \cD \times \cY \times \cX\to \R$. Then the following are equivalent:
\begin{enumerate}
    \item[\textnormal{(a)}] For all $d, d^\prime \in \cD \setminus \{k\}$ the difference $\omega_k(d, y, \bx) - \omega_k(d^\prime, y, \bx)$ does not depend on $y \in \cY$.
    \item[\textnormal{(b)}] There exist functions $f_k: \cD\setminus \{k\} \times \cX \to \R$ and $g_k: \cY \times \cX \to \R$ such that $\omega_k(d, y, \bx) = f_k(d, \bx) + g_k(y, \bx)$ for all $d \in \cD \setminus \{k\}$ and $y \in \cY$.
\end{enumerate}
\end{lemma}
\begin{proof}
    We first prove that (a) implies (b). Fix $d_0 \in \cD \setminus \{k\}$ and define $f_k(d, \bx) := \omega_k(d, y, \bx) - \omega_k(d_0, y, \bx)$ for any fixed $y \in \cY$. By assumption, $f_k(d, \bx)$ does not depend on $y$ for $d \neq k$. Next, define $g_k(y, \bx) := \omega_k(d_0, y, \bx)$. Then
    \[
        \omega_k(d, y, \bx) = f_k(d, \bx) + \omega_k(d_0, y, \bx) = f_k(d, \bx) + g_k(y, \bx),
    \]
    for $d \neq k$, which proves the claim.
    We next prove that (b) implies (a). If $\omega_k(d, y, \bx) =
    f_k(d, \bx) + g_k(y, \bx)$ for all $d \in \cD \setminus \{k\}$,
    then for any $d, d^\prime \in \cD \setminus \{k\}$ we have
    \[
      \omega_k(d, y, \bx) -   \omega_k(d^\prime, y, \bx) = f_k(d, \bx) + g_k(y, \bx) - (f_k(d^\prime, \bx) + g_k(y, \bx)) = f_k(d, \bx) - f_k(d^\prime, \bx),
    \]
    which is independent of $y$. This proves the claim.
\end{proof}

\section{Derivation of Weights}\label{app:weights}
We illustrate the results from Section~\ref{sec:symbolic} by working out the linear systems in Equations~\eqref{eq:linear_system}, \eqref{eq:linear_system_exact}, and \eqref{eq:linear_system_std} for the binary case \( D, Y \in \{0,1\} \). We suppress covariates for notational simplicity though all quantities may be read pointwise in $\bx$. In this setting, the loss can take eight distinct values $\ell(d; y_0, y_1)$, corresponding to two decisions across four principal strata.

Recall that we seek to solve the linear system
\begin{align*}
    \bA^b \bw^b = \bm{\ell},
\end{align*}
where $b \in \{\std, \exa, \add\}$ indexes the standard, exactly identifiable, and additive cases, respectively, as described in Section~\ref{sec:symbolic}.  
For $b = \add$, $\bA^\add \in \R^{8 \times 12}$ and the system is given by
\[
\bA^\add =
\left[
\begin{array}{cccccccccccc}
1 & 0 & 0 & 0 & 0 & 0 & 1 & 0 & 1 & 0 & 0 & 0 \\
0 & 0 & 1 & 0 & 1 & 0 & 0 & 0 & 1 & 0 & 0 & 0 \\
0 & 1 & 0 & 0 & 0 & 0 & 1 & 0 & 0 & 0 & 1 & 0 \\
0 & 0 & 1 & 0 & 0 & 1 & 0 & 0 & 0 & 0 & 1 & 0 \\
1 & 0 & 0 & 0 & 0 & 0 & 0 & 1 & 0 & 1 & 0 & 0 \\
0 & 0 & 0 & 1 & 1 & 0 & 0 & 0 & 0 & 1 & 0 & 0 \\
0 & 1 & 0 & 0 & 0 & 0 & 0 & 1 & 0 & 0 & 0 & 1 \\
0 & 0 & 0 & 1 & 0 & 1 & 0 & 0 & 0 & 0 & 0 & 1
\end{array}
\right], \qquad
    \bw^\add =
\begin{bmatrix}
\omega_0(0,0) \\
\omega_0(0,1) \\
\omega_1(1,0) \\
\omega_1(1,1) \\
\omega_0(1,0) \\
\omega_0(1,1) \\
\omega_1(0,0) \\
\omega_1(0,1) \\
\varpi(0,0) \\
\varpi(0,1) \\
\varpi(1,0) \\
\varpi(1,1)
\end{bmatrix},
\qquad
    \bm{\ell} =
\begin{bmatrix}
\ell(0; 0,0)\\[5pt]
\ell(1; 0,0)\\[5pt]
\ell(0; 1,0)\\[5pt]
\ell(1; 1,0)\\[5pt]
\ell(0; 0,1)\\[5pt]
\ell(1; 0,1)\\[5pt]
\ell(0; 1,1)\\[5pt]
\ell(1; 1,1)
\end{bmatrix}.
\]
The other cases are obtained by deleting columns from \(\bA^\add\):
\[
\bA^\exa = \bA^\add_{[:, 1:8]}, \quad \bw^\exa = \bw^\add_{1:8}, 
\qquad
\bA^\std = \bA^\add_{[:, 1:4]}, \quad \bw^\std = \bw^\add_{1:4},
\]
where \(\bA^\add_{[:, 1:k]}\) denotes the first \(k\) columns of matrix \(\bA^\add\), and \(\bw^\add_{1:k}\) the first \(k\) entries of vector \(\bw^\add\). Thus the exact system removes the intercept terms while the standard system keeps only the diagonal weights.
Since $\bA^b$ is fixed and known, we can solve the systems of linear equations either symbolically or numerically, depending on the representation of $\bm{\ell}$.

Furthermore, since $\bA^b$ is given, its left null space can be represented by a matrix $\bm{N}^b$ whose columns form a basis for $\ker\!\left((\bA^b)^\top\right)$. By standard linear algebra, $\bm{\ell} \in \mathrm{im}(\bA^b)$ (i.e., the system admits a solution) if and only if
\[
(\bm{N}^b)^\top \bm{\ell} = \bm{0}.
\]
This directly yields the equality constraints on $\ell(d; y_0,y_1)$ needed for identification in terms of principal strata.

In the standard case, i.e., $b = \std$, the null space can be represented as 
\[
\bm{N}^\std =
\left[
\begin{array}{cccc}
0 & -1 & 0 & 0  \\
-1 & 0 & 0 & 0 \\
0 & 0 & -1 & 0 \\
1 & 0 & 0 & 0 \\
0 & 1 & 0 & 0 \\
0 & 0 & 0 & -1 \\
0 & 0 & 1 & 0 \\
0 & 0 & 0 & 1
\end{array}
\right],
\]
with $\operatorname{rank}(\bm{N}^\std) = 4$. This recovers the four constraints of Equation~\eqref{eq:standard}, leaving four degrees of freedom to choose the loss $\ell$, as indicated in Table~\ref{tab:identifiability_binary_case}.

In the exactly identifiable case, i.e., $b = \exa$, the null space takes the form
\[
\bm{N}^\exa =
\left[
\begin{array}{cc}
1 & 0 \\
0 & 1 \\
-1 & 0 \\
0 & -1 \\
-1 & 0 \\
0 & -1 \\
1 & 0 \\
0 & 1 
\end{array}
\right],
\]
with $\operatorname{rank}(\bm{N}^\exa) = 2$, yielding the two constraints of Equation~\eqref{eq:exact} and leaving six degrees of freedom for the loss.

Finally, in the general additive case with $b = \add$, this reduces to
\[
\bm{N}^\add =
\left[
\begin{array}{c}
-1 \\
1  \\
1 \\
-1 \\
1 \\
-1 \\
-1 \\
1 
\end{array}
\right],
\]
with $\operatorname{rank}(\bm{N}^\add) = 1$, giving the single restriction of Equation~\eqref{eq:constant} and leaving seven degrees of freedom.

The above derivation illustrates how the identification conditions translate into explicit equality constraints on the loss values $\ell(d; y_0,y_1)$ in terms of principal strata.

\end{document}